\documentclass[a4paper,leqno]{article}

\usepackage{amsmath}
\usepackage{amsthm}
\usepackage{amssymb}
\usepackage{hyperref}
\usepackage{caption}
\usepackage{subcaption}
\usepackage{graphicx}
\usepackage{lineno}

\newtheorem{theorem}{Theorem}
\newtheorem{lemma}[theorem]{Lemma}
\newtheorem{proposition}[theorem]{Proposition}

\newtheorem{conjecture}[theorem]{Conjecture}
\newtheorem{corollary}[theorem]{Corollary}
\newtheorem{definition}[theorem]{Definition}
\newtheorem{problem}[theorem]{Problem}
\newtheorem{algorithm}[theorem]{Algorithm}
\newtheorem{procedure}[theorem]{Procedure}
\newtheorem{example}[theorem]{Example}
\newtheorem*{remark}{Remark}
\numberwithin{equation}{section}
\numberwithin{theorem}{section}

\newcommand{\abs}[1]{\vert{#1}\vert}
\newcommand{\set}[1]{\left\{#1\right\}}

\newcommand{\uni}[2]{{#1}\big\vert_{#2}}

\newcommand{\wH}[3]{\widetilde{\mathcal{H}}(\uni{#1}{#2}; {#3})}
\newcommand{\wFv}[3]{\widetilde{F}_v(\uni{#1}{#2}; {#3})}

\DeclareMathOperator{\G}{G}
\DeclareMathOperator{\V}{V}
\DeclareMathOperator{\E}{E}
\DeclareMathOperator{\N}{N}

\modulolinenumbers[1]

\begin{document}
	
	
\title{The vertex Folkman numbers\\ $F_v(a_1, ..., a_s; m - 1) = m + 9$,\\ if $\max\{a_1, ..., a_s\} = 5$}
		
\author{
	Aleksandar Bikov\thanks{Corresponding author} \hfill Nedyalko Nenov\\
	\let\thefootnote\relax\footnote{Email addresses: \texttt{asbikov@fmi.uni-sofia.bg}, \texttt{nenov@fmi.uni-sofia.bg}}\\
	\smallskip\\
	Faculty of Mathematics and Informatics\\
	Sofia University "St. Kliment Ohridski"\\
	5, James Bourchier Blvd.\\
	1164 Sofia, Bulgaria
}

\maketitle

\begin{abstract}
For a graph $G$ the expression $G \overset{v}{\rightarrow} (a_1, ..., a_s)$ means that for any $s$-coloring of the vertices of $G$ there exists $i \in \set{1, ..., s}$ such that there is a monochromatic $a_i$-clique of color $i$. The vertex Folkman numbers
$$F_v(a_1, ..., a_s; m - 1) = \min\{\vert\V(G)\vert : G \overset{v}{\rightarrow} (a_1, ..., a_s) \mbox{ and } K_{m - 1} \not\subseteq G\}.$$
are considered, where $m = \sum_{i = 1}^{s}(a_i - 1) + 1$.

With the help of computer we show that $F_v(2, 2, 5; 6) = 16$ and then we prove
$$F_v(a_1, ..., a_s; m - 1) = m + 9,$$
if $\max\{a_1, ..., a_s\} = 5$.

We also obtain the bounds
$$m + 9 \leq F_v(a_1, ..., a_s; m - 1) \leq m + 10,$$
if $\max\set{a_1, ..., a_s} = 6$.

\bigskip

\end{abstract}
		


\section{Introduction}

In this paper only finite, non-oriented graphs without loops and multiple edges are considered. The following notations are used:

$\V(G)$ - the vertex set of $G$;

$\E(G)$ - the edge set of $G$;

$\overline{G}$ - the complement of $G$;

$\omega(G)$ - the clique number of $G$;

$\alpha(G)$ - the independence number of $G$;

$\N(v), \N_G(v), v \in \V(G)$ - the set of all vertices of G adjacent to $v$;

$d(v), v \in \V(G)$ - the degree of the vertex $v$, i.e. $d(v) = \abs{\N(v)}$;

$G-v, v \in \V(G)$ - subgraph of $G$ obtained from $G$ by deleting the vertex $v$ and all edges incident to $v$;

$G-e, e \in \E(G)$ - subgraph of $G$ obtained from $G$ by deleting the edge $e$;

$G+e, e \in \E(\overline{G})$ - supergraph of G obtained by adding the edge $e$ to $\E(G)$.

$K_n$ - complete graph on $n$ vertices;

$C_n$ - simple cycle on $n$ vertices;

Let $G_1$ and $G_2$ be two graphs without common vertices. Denote by $G_1+G_2$ the graph $G$ for which: $\V(G) = \V(G_1) \cup \V(G_2)$ and $\E(G) = \E(G_1) \cup \E(G_2) \cup E'$, where $E' = \set{[x, y] : x \in \V(G_1), y \in \V(G_2)}$, i.e. $G$ is obtained by making every vertex of $G_1$ adjacent to every vertex of $G_2$.

All undefined terms can be found in \cite{W01}.\\

Let $a_1, ..., a_s$ be positive integers. The expression $G \overset{v}{\rightarrow} (a_1, ..., a_s)$ means that for any coloring of $\V(G)$ in $s$ colors ($s$-coloring) there exists $i \in \set{1, ..., s}$ such that there is a monochromatic $a_i$-clique of color $i$. In particular, $G \overset{v}{\rightarrow} (a_1)$ means that $\omega(G) \geq a_1$. Further, for convenience instead of $G \overset{v}{\rightarrow} (\underbrace{2, ..., 2}_r)$ we write $G \overset{v}{\rightarrow} (2_r)$ and instead of $G \overset{v}{\rightarrow} (\underbrace{2, ..., 2}_r, a_1, ..., a_s)$ we write $G \overset{v}{\rightarrow} (2_r, a_1, ..., a_s)$.

Define:

$\mathcal{H}(a_1, ..., a_s; q) = \set{ G : G \overset{v}{\rightarrow} (a_1, ..., a_s) \mbox{ and } \omega(G) < q }.$

$\mathcal{H}(a_1, ..., a_s; q; n) = \set{ G : G \in \mathcal{H}(a_1, ..., a_s; q) \mbox{ and } \abs{\V(G)} = n }.$

The vertex Folkman number $F_v(a_1, ..., a_s; q)$ is defined by the equality:

$F_v(a_1, ..., a_s; q) = \min\set{\abs{\V(G)} : G \in \mathcal{H}(a_1, ..., a_s; q)}$.

The graph $G$ is called an extremal graph in $\mathcal{H}(a_1, ..., a_s; q)$ if $G \in \mathcal{H}(a_1, ..., a_s;q )$ and $\abs{\V(G)} = F_v(a_1, ..., a_s; q)$. The set of all extremal graphs in $\mathcal{H}(a_1, ..., a_s; q)$ is denoted by $\mathcal{H}_{extr}(a_1, ..., a_s; q)$ .

The graph $G$ is called a maximal graph in $\mathcal{H}(a_1, ..., a_s; q)$ if $G \in \mathcal{H}(a_1, ..., a_s; q)$, but $G + e \not\in \mathcal{H}(a_1, ..., a_s; q), \forall e \in \E(\overline{G})$, i.e. $\omega(G + e) \geq q, \forall e \in \E(\overline{G})$. Also, $G$ is called a minimal graph in $\mathcal{H}(a_1, ..., a_s; q)$ if $G \in \mathcal{H}(a_1, ..., a_s; q)$, but $G - e \not\in \mathcal{H}(a_1, ..., a_s; q), \forall e \in \E(G)$, i.e. $G - e \overset{v}{\nrightarrow} (a_1, ..., a_s), \forall e \in \E(G)$. If $G$ is both maximal and minimal graph in $\mathcal{H}(a_1, ..., a_s; q)$, then we say that $G$ is a bicritical graph in $\mathcal{H}(a_1, ..., a_s; q)$. 

Folkman proves in \cite{Fol70} that:
\begin{equation}
\label{equation: F_v(a_1, ..., a_s; q) exists}
F_v(a_1, ..., a_s; q) \mbox{ exists } \Leftrightarrow q > \max\set{a_1, ..., a_s}.
\end{equation}
Other proofs of (\ref{equation: F_v(a_1, ..., a_s; q) exists}) are given in \cite{DR08} and \cite{LRU01}.\\
Obviously $F_v(a_1, ..., a_s; q)$ is a symmetric function of $a_1, ..., a_s$ and if $a_i = 1$, then
\begin{equation*}
F_v(a_1, ..., a_s; q) = F_v(a_1, ..., a_{i-1}, a_{i+1}, ..., a_s; q).
\end{equation*}
Therefore, it is enough to consider only such Folkman numbers $F_v(a_1, ..., a_s; q)$ for which
\begin{equation}
\label{equation: 2 leq a_1 leq ... leq a_s}
2 \leq a_1 \leq ... \leq a_s
\end{equation}
We call the numbers $F_v(a_1, ..., a_s; q)$ for which the inequalities (\ref{equation: 2 leq a_1 leq ... leq a_s}) hold, canonical vertex Folkman numbers.\\
In \cite{LU96} for arbitrary positive integers $a_1, ..., a_s$ are defined
\begin{equation}
\label{equation: m and p}
m(a_1, ..., a_s) = m = \sum\limits_{i=1}^s (a_i - 1) + 1 \quad \mbox{ and } \quad p = \max\set{a_1, ..., a_s}.
\end{equation}
Obviously $K_m \overset{v}{\rightarrow} (a_1, ..., a_s)$ and $K_{m - 1} \overset{v}{\nrightarrow} (a_1, ..., a_s)$. Therefore,
\begin{equation*}
F_v(a_1, ..., a_s; q) = m, \quad q \geq m + 1.
\end{equation*}
In accordance with (\ref{equation: F_v(a_1, ..., a_s; q) exists}),
\begin{equation*}
F_v(a_1, ..., a_s; m) \mbox{ exists } \Leftrightarrow m \geq p + 1.
\end{equation*}
For these numbers the following theorem is true:

\begin{theorem}
\label{theorem: F_v(a_1, ..., a_s; m) = m + p}
Let $a_1, ..., a_s$ be positive integers and $m$ and $p$ are defined by (\ref{equation: m and p}). If $m \geq p + 1$, then:
\begin{flalign*}
F_v(a_1, ..., a_s; m) = m + p, \ \mbox{\cite{LU96},\cite{LRU01}}. && \tag{a} 
\end{flalign*}
\begin{flalign*}
K_{m+p} - C_{2p + 1} = K_{m - p - 1} + \overline{C}_{2p + 1} && \tag{b}
\end{flalign*}
is the only extremal graph in $\mathcal{H}(a_1, ..., a_s; m)$, \ \cite{LRU01}.
\end{theorem}

Other proofs of Theorem \ref{theorem: F_v(a_1, ..., a_s; m) = m + p} are given in \cite{Nen00} and \cite{Nen01}.\\
In accordance with (\ref{equation: F_v(a_1, ..., a_s; q) exists}),
\begin{equation}
\label{equation: F_v(a_1, ..., a_s; m - 1) exists}
F_v(a_1, ..., a_s; m - 1) \mbox{ exists } \Leftrightarrow m \geq p + 2.
\end{equation}
Let $m$ and $p$ be defined by (\ref{equation: m and p}). Then
\begin{equation}
\label{equation: F_v(a_1, ..., a_s, m - 1) = ...}
F_v(a_1, ..., a_s, m - 1) = \begin{cases}
m + 4, & \emph{if $p = 2$ and $m \geq 6$, \cite{Nen83}}\\
m + 6, & \emph{if $p = 3$ and $m \geq 6$, \cite{Nen02}}\\
m + 7, & \emph{if $p = 4$ and $m \geq 6$, \cite{Nen02}}.
\end{cases}
\end{equation}
The remaining canonical numbers $F_v(a_1, ..., a_s; m - 1)$, $p \leq 4$ are: $F_v(2, 2, 2; 3) = 11$, \cite{Myc55} and \cite{Chv79}, $F_v(2, 2, 2, 2; 4) = 11$, \cite{Nen84} (see also \cite{Nen98}), $F_v(2, 2, 3; 4) = 14$, \cite{Nen00} and \cite{CR06}, $F_v(3, 3; 4) = 14$, \cite{Nen81} and \cite{PRU99}. From these facts it becomes clear that we know all Folkman numbers of the form $F_v(a_1, ..., a_s; m - 1)$ when $\max\set{a_1, ..., a_s} \leq 4$. The only known canonical vertex Folkman number of the form $F_v(a_1, ..., a_s, m - 1)$, $p \geq 5$ is $F_v(3, 5; 6) = 16$, \cite{SLPX12}. Since we know all the numbers $F_v(a_1, ..., a_s; m - 1)$ when $p = 2$, further we shall assume that $p \geq 3$. The following bounds for these numbers are known
\begin{equation}
\label{equation: m + p + 2 leq F_v(a_1, ..., a_s; m - 1) leq m + 3p}
m + p + 2 \leq F_v(a_1, ..., a_s; m - 1) \leq m + 3p, \quad p\geq 3.
\end{equation}
The lower bound is obtained in \cite{Nen00} and the upper bound is obtained in \cite{KN06a}. It is easy to see that in the border case $m = p + 2$ when $p \geq 3$ there are only two canonical numbers of the form $F_v(a_1, ..., a_s; m - 1)$, namely $F_v(2, 2, p; p + 1)$ and $F_v(3, p; p+1)$.\\
With the help of the numbers $F_v(2, 2, p; p + 1)$, the lower bound in (\ref{equation: m + p + 2 leq F_v(a_1, ..., a_s; m - 1) leq m + 3p}) can be improved.

\begin{theorem}
\label{theorem: F_v(a_1, ..., a_s; m - 1) geq m + p + 3}
\cite{Nen02}
Let $a_1, ..., a_s$ be positive integers, $m$ and $p$ are defined by (\ref{equation: m and p}), $p \geq 3$ and $m \geq p + 2$. If $F_v(2, 2, p; p + 1) \geq 2p + 5$, then
\begin{equation*}
F_v(a_1, ..., a_s; m - 1) \geq m + p + 3.
\end{equation*}
\end{theorem}

Some graphs, with which upper bounds for $F_v(3, p; p + 1)$ are obtained, can be used for obtaining general upper bounds for $F_v(a_1, ..., a_s; m - 1)$. For example, the graph $\Gamma_p$ from \cite{Nen00}, which witnesses the bound $F_v(3, p; p + 1) \leq 4p + 2, p \geq 3$, helps to obtain the upper bound in (\ref{equation: m + p + 2 leq F_v(a_1, ..., a_s; m - 1) leq m + 3p}). Thus, obtaining bounds for the numbers $F_v(a_1, ..., a_s; m -1)$ and computing some of them is related with computation and obtaining bounds for the numbers $F_v(2, 2, p; p + 1)$ and $F_v(3, p; p + 1)$.
It is easy to see that $G \overset{v}{\rightarrow} (3, p) \Rightarrow G \overset{v}{\rightarrow} (2, 2, p)$. Therefore, the following inequality holds:
\begin{equation}
\label{equation: F_v(2, 2, p; p + 1) leq F_v(3, p; p + 1)}
F_v(2, 2, p; p + 1) \leq F_v(3, p; p + 1), \quad p \geq 3.
\end{equation}
However, we do not know if these numbers are different. For now we do know that they are equal if $p = 3$ or $p = 4$.

\begin{problem}
\label{problem: F_v(2, 2, p; p + 1) < F_v(3, p; p + 1)}
\cite{KN06c}
Does there exist a positive integer $p$ for which the inequality (\ref{equation: F_v(2, 2, p; p + 1) leq F_v(3, p; p + 1)}) is strict?
\end{problem}

In this paper we prove $F_v(2, 2, 5; 6) = 16$ and a new proof of the equality $F_v(3, 5; 6) = 16$, \cite{SLPX12} is given. Hence for $p = 5$ there is an equality in (\ref{equation: F_v(2, 2, p; p + 1) leq F_v(3, p; p + 1)}). We find all extremal graphs in $\mathcal{H}(2, 2, 5; 6)$ and in $\mathcal{H}(3, 5; 6)$. With the help of these results, we compute the numbers $F_v(a_1, ..., a_s; m - 1) = m + 9$, when $\max\set{a_1, ..., a_s} = 5$. In the case $\max\set{a_1, ..., a_s} = 6$ we improve the bounds (\ref{equation: m + p + 2 leq F_v(a_1, ..., a_s; m - 1) leq m + 3p}) by proving $m + 9 \leq F_v(a_1, ..., a_s; m - 1) \leq m + 10$. The exact formulations of the obtained results are as follows:

\begin{theorem}
\label{theorem: abs(mathcal(H)(2, 2, 5; 6; 16)) = 147}
$\abs{\mathcal{H}(2, 2, 5; 6; 16)} = 147$. Some properties of the graphs in $\mathcal{H}(2, 2, 5; 6; 16)$ are listed in Table \ref{table: H(2, 2, 5; 6; 16) properties}. Among them there are 4 bicritical graphs, which are shown in Figure \ref{figure: H(2, 2, 5; 6; 16) bicritical}, and some of their properties are listed in Table \ref{table: H(2, 2, 5; 6; 16) bicritical properties}.
\end{theorem}

\begin{theorem}
\label{theorem: F_v(2, 2, 5; 6) = 16}
$F_v(2, 2, 5; 6) = 16$ and the graphs from Theorem \ref{theorem: abs(mathcal(H)(2, 2, 5; 6; 16)) = 147} are all the graphs in $\mathcal{H}_{extr}(2, 2, 5; 6)$.
\end{theorem}

\begin{corollary}
\label{corollary: F_v(3, 5; 6) = 16}
\cite{SLPX12}
$F_v(3, 5; 6) = 16$.
\end{corollary}

\begin{proof}
From Theorem \ref{theorem: F_v(2, 2, 5; 6) = 16} and (\ref{equation: F_v(2, 2, p; p + 1) leq F_v(3, p; p + 1)}) we obtain $F_v(3, 5; 6) \geq 16$. Since among the graphs from Theorem \ref{theorem: abs(mathcal(H)(2, 2, 5; 6; 16)) = 147} there are such, which belong to $\mathcal{H}(3, 5; 6)$ (see Figure \ref{figure: H(3, 5; 6; 16)}), it follows that $F_v(3, 5; 6) \leq 16$.
\end{proof}

\begin{theorem}
\label{theorem: abs(mathcal(H)(3, 5; 6; 16)) = 4}
$\abs{\mathcal{H}(3, 5; 6; 16)} = 4$. The graphs from $\mathcal{H}(3, 5; 6; 16)$ are shown in Figure \ref{figure: H(3, 5; 6; 16)} and some of their properties are listed in Table \ref{table: H(3, 5; 6; 16) properties}.
\end{theorem}

\begin{theorem}
\label{theorem: F_v(a_1, ..., a_s; m - 1) = m + 9, max set(a_1, ..., a_s) = 5}
Let $a_1, ..., a_s$ be positive integers, $m = \sum\limits_{i=1}^s (a_i - 1) + 1$, $\max\set{a_1, ..., a_s} = 5$ and $m \geq 7$. Then:
\begin{equation*}
F_v(a_1, ..., a_s; m - 1) = m + 9
\end{equation*}
\end{theorem}

At the end of this paper as a consequence of these results and with the help of one graph (see Figure \ref{figure: H(3, 6; 7; 18) cap H(4, 5; 7; 18)}) from \cite{SXP09} we prove that

\begin{theorem}
\label{theorem: m + 9 leq F_v(a_1, ..., a_s; m - 1) leq m + 10, max set(a_1, ..., a_s) = 6}
Let $a_1, ..., a_s$ be positive integers, $m = \sum\limits_{i=1}^s (a_i - 1) + 1$, $\max\set{a_1, ..., a_s} = 6$ and $m \geq 8$. Then:\\

$m + 9 \leq F_v(a_1, ..., a_s; m - 1) \leq m + 10$
\end{theorem}

\begin{remark}
According to (\ref{equation: F_v(a_1, ..., a_s; m - 1) exists}) the conditions $m \geq 7$ in Theorem \ref{theorem: F_v(a_1, ..., a_s; m - 1) = m + 9, max set(a_1, ..., a_s) = 5} and $m \geq 8$ in Theorem \ref{theorem: m + 9 leq F_v(a_1, ..., a_s; m - 1) leq m + 10, max set(a_1, ..., a_s) = 6} are necessary.
\end{remark}

\section{Proof of Theorem \ref{theorem: abs(mathcal(H)(2, 2, 5; 6; 16)) = 147}}

We adapt Algorithm A1 from \cite{PRU99} to obtain all graphs in $\mathcal{H}(2, 2, 5; 6; 16)$ with the help of computer. Similar algorithms are used in \cite{CR06}, \cite{XLS10}, \cite{LR11} and \cite{SLPX12}. Also, with the help of computer, results for Folkman numbers are obtained in \cite{JR95}, \cite{SXP09}, \cite{SXL09} and \cite{DLSX13}.\\
The naive approach for finding all graphs in $\mathcal{H}(2, 2, 5; 6; 16)$ suggests to check all graphs of order 16 for inclusion in $\mathcal{H}(2, 2, 5; 6)$. However, this is practically impossible because the number of graphs to check is too large. The method that is described uses an algorithm for effective generation of all maximal graphs in $\mathcal{H}(2, 2, 5; 6; 16)$. The other graphs in $\mathcal{H}(2, 2, 5; 6; 16)$ are their subgraphs. The algorithm is based on the following proposition:

\begin{proposition}
\label{proposition: forming all maximal graphs in mathcal(H)(2_r, p; q; n)}
Let $G$ be a maximal graph in $\mathcal{H}(2_r, p; q; n)$ and $v_1, v_2, ..., v_k$ are independent vertices. Let $H = G - \set{v_1, v_2, ..., v_k}$. Then:
\begin{flalign*}
H \in \mathcal{H}(2_{r-1}, p; q; n - k) && \tag{a}
\end{flalign*}
\begin{flalign*}
\mbox{the addition of a new edge to $H$ forms a new $(q - 1)$-clique} && \tag{b}
\end{flalign*}
\begin{flalign*}
\mbox{$\N_G(v_i)$ is a maximal $K_{q - 1}$-free subset} && \tag{c}
\end{flalign*}
of $\V(H)$, $i = 1, ..., k$
\end{proposition}

\begin{proof}
The proposition (a) follows from the assumption that $G \in \mathcal{H}(2_r, p; q; n)$, (b) and (c) follow from the maximality of $G$.
\end{proof}

The following algorithm, which is a modification of Algorithm А1 from \cite{PRU99}, generates all maximal graphs in $\mathcal{H}(2_r, p; q; n)$ with independence number at least $k$:

\begin{algorithm}
\label{algorithm: add_is_kn_free}
Generation of all maximal graphs in $\mathcal{H}(2_r, p; q; n)$ with independence number at least $k$ by adding $k$ independent vertices to the graphs from $\mathcal{H}(2_{r-1}, p; q; n - k)$ in which the addition of a new edge forms a new $(q - 1)$-clique.

1. Let $\mathcal{A} \subseteq \mathcal{H}(2_{r-1}, p; q; n - k)$ is the set of these graphs in which the addition of a new edge forms a new $(q - 1)$-clique (see Proposition \ref{proposition: forming all maximal graphs in mathcal(H)(2_r, p; q; n)} (a) and (b)). The maximal graphs in $\mathcal{H}(2_r, p; q; n)$ are output in $\mathcal{B}$.

2. For each graph $H \in \mathcal{A}$:

2.1. Find the family $\mathcal{M}(H) = \set{M_1, ..., M_t}$ of all maximal $K_{q - 1}$-free subsets of $\V(H)$.

2.2. Consider all the $k$-tuples $(M_{i_1}, M_{i_2}, ..., M_{i_k})$ of elements of $\mathcal{M}(H)$ for which $1 \leq i_1 \leq ... \leq i_k \leq t$ (in these $k$-tuples some subsets $M_i$ can coincide). For every such $k$-tuple construct the graph $G = G(M_{i_1}, M_{i_2}, ..., M_{i_k})$ by adding to $\V(H)$ new independent vertices $v_1, v_2, ..., v_k$, so that $N_G(v_j) = M_{i_j}, j = 1, ..., k$ (see Proposition \ref{proposition: forming all maximal graphs in mathcal(H)(2_r, p; q; n)} (c)). If $\omega(G + e) = q, \forall e \in \E(\overline{G})$, then add $G$ to $\mathcal{B}$.

3. Exclude the isomorph copies of graphs from $\mathcal{B}$.

4. Exclude from $\mathcal{B}$ all graph which are not in $\mathcal{H}(2_r, p; q; n)$.

\end{algorithm}

According to Proposition \ref{proposition: forming all maximal graphs in mathcal(H)(2_r, p; q; n)}, at the end of step 4 $\mathcal{B}$ is the set of all maximal graphs in $\mathcal{H}(2_r, p; q; n)$.

Intermediate problems, that are solved, are finding all graphs in $\mathcal{H}(2, 5; 6; 13)$ and in $\mathcal{H}(5; 6; 10)$. For each of the sets $\mathcal{H}(2, 2, 5; 6; 16)$ and $\mathcal{H}(2, 5; 6; 13)$ we start by finding the maximal graphs in them. The remaining graphs are obtained by removing edges from the maximal graphs. Using Algorithm \ref{algorithm: add_is_kn_free} we can obtain the maximal graphs in $\mathcal{H}(2, 2, 5; 6; 16)$ with independence number at least 3 by adding 3 independent vertices to graphs in $\mathcal{H}(2, 5; 6; 13)$. Similarly, we can obtain the maximal graphs in $\mathcal{H}(2, 5; 6; 13)$ with independence number at least 3 by adding 3 independent vertices to graphs in $\mathcal{H}(5; 6; 10)$. What remains is to find the maximal graphs in these sets with independence number 2. Let

$\mathcal{R}(p, q) = \set{G : \alpha(G) < p \mbox{ and } \omega(G) < q}$

$\mathcal{R}(p, q; n) = \set{G : G \in \mathcal{R}(p, q) \mbox{ and } \abs{\V(G)} = n}$

The graphs $\mathcal{R}(3, 6)$ are known (see \cite{McK_r} and \cite{Rad06}). The maximal graphs in $\mathcal{H}(2, 2, 5; 6; 16)$ with independence number 2 are a subset of $\mathcal{R}(3, 6; 16)$ and the maximal graphs in $\mathcal{H}(2, 5; 6; 13)$ with independence number 2 are a subset of $\mathcal{R}(3, 6; 13)$

The \emph{nauty} programs \cite{McK90} have an important role in this work. We use them for fast generation of non-isomorphic graphs, isomorph rejection and to determine the automorphism groups of graphs.

\subsection{Finding all graphs in $\mathcal{H}(5; 6; 10)$}

It is clear that $\mathcal{H}(5; 6; 10)$ is the set of 10 vertex graphs with clique number 5. The number of non-isomorphic graphs of order 10 is 12 005 168. Out of those we can easily find the graphs with clique number 5. Thus, we obtain all 1 724 440 graphs in $\mathcal{H}(5; 6; 10)$. 

\subsection{Finding all graphs in $\mathcal{H}(2, 5; 6; 13)$}

\begin{algorithm}
\label{algorithm: H(2, 5; 6; 13)}
Finding all graphs in $\mathcal{H}(2, 5; 6; 13)$.

1. Find all maximal graphs $G \in \mathcal{H}(2, 5; 6; 13)$ for which $\alpha(G)\geq 3$:

1.1. Determine which of the graphs in $\mathcal{H}(5; 6; 10)$ have the property that the addition of a new edge forms a new 5-clique.

1.2. Using Algorithm \ref{algorithm: add_is_kn_free} add three independent vertices to the graphs from step 1.1. to obtain the graphs wanted in step 1.

2. Find all maximal graphs $G \in \mathcal{H}(2, 5; 6; 13)$ for which $\alpha(G) = 2$:

2.1. In order to do so, check which of the graphs in $\mathcal{R}(3, 6; 13)$ are maximal graphs in $\mathcal{H}(2, 5; 6; 13)$.

3. The union of the graphs from steps 1. and 2. gives all maximal graphs in $\mathcal{H}(2, 5; 6; 13)$. By removing edges from them the remaining graphs in $\mathcal{H}(2, 5; 6; 13)$ are obtained.

\end{algorithm}

Results of computations:

Step 1: Among all the graphs in $\mathcal{H}(5; 6; 10)$ exactly 3633 have the property that the addition of a new edge forms a new 5-clique. By adding three independent vertices to them we obtain 326 maximal graphs in $\mathcal{H}(2, 5; 6; 13)$.

Step 2: The number of graphs in $\mathcal{R}(3, 6; 13)$ is 275 086 \cite{McK_r36_13}. Among them 61 are maximal graphs in $\mathcal{H}(2, 5; 6; 13)$.

Step 3: The union of the graphs from steps 1. and 2. gives all 387 maximal graphs in $\mathcal{H}(2, 5; 6; 13)$. By removing edges from them we obtain all 20 013 726 graphs in $\mathcal{H}(2, 5; 6; 13)$ . 

\subsection{Finding all graphs in $\mathcal{H}(2, 2, 5; 6; 16)$}

\begin{algorithm}
\label{algorithm: H(2, 2, 5; 6; 16)}
Finding all graphs in $\mathcal{H}(2, 2, 5; 6; 16)$.

1. Find all maximal graphs $G \in \mathcal{H}(2, 2, 5; 6; 16)$ for which $\alpha(G)\geq 3$:

1.1. Determine which of the graphs in $\mathcal{H}(2, 5; 6; 13)$ have the property that the addition of a new edge forms a new 5-clique.

1.2. Using Algorithm \ref{algorithm: add_is_kn_free} add three independent vertices to the graphs from step 1.1. to obtain the graphs wanted in step 1.

2. Find all maximal graphs $G \in \mathcal{H}(2, 2, 5; 6; 16)$ for which $\alpha(G) = 2$:

2.1. In order to do so, check which of the graphs in $\mathcal{R}(3, 6; 16)$ are maximal graphs in $\mathcal{H}(2, 2, 5; 6; 16)$.

3. The union of the graphs from steps 1. and 2. gives all maximal graphs in $\mathcal{H}(2, 2, 5; 6; 16)$. By removing edges from them the remaining graphs in $\mathcal{H}(2, 2, 5; 6; 16)$ are obtained.

\end{algorithm}

\begin{figure}
	\centering
	\begin{subfigure}[b]{0.5\textwidth}
		\centering
		\includegraphics[height=240px,width=120px]{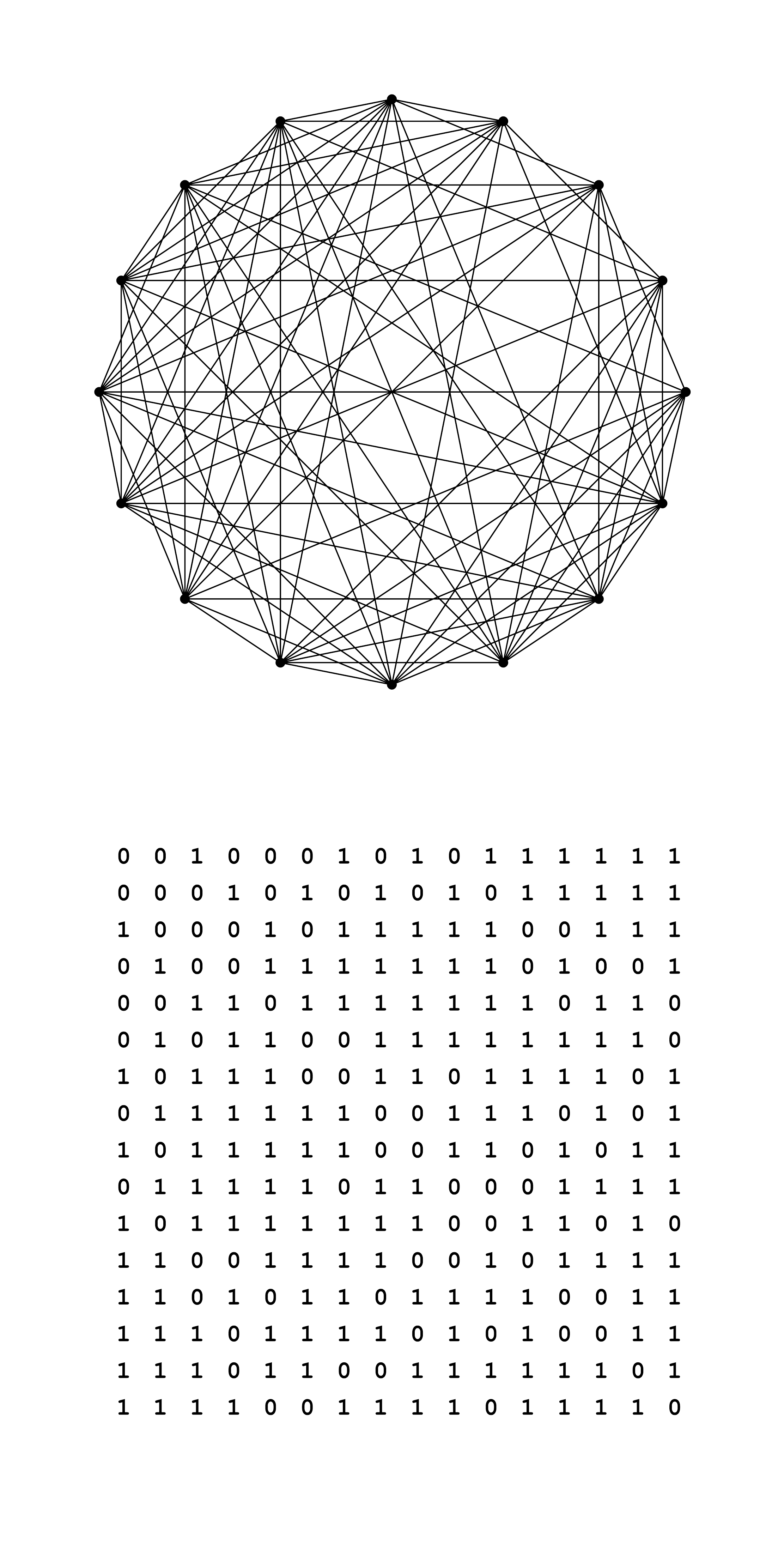}
		\caption*{\emph{$G_{74}$}}
		\label{figure: G_74}
	\end{subfigure}%
	\begin{subfigure}[b]{0.5\textwidth}
		\centering
		\includegraphics[height=240px,width=120px]{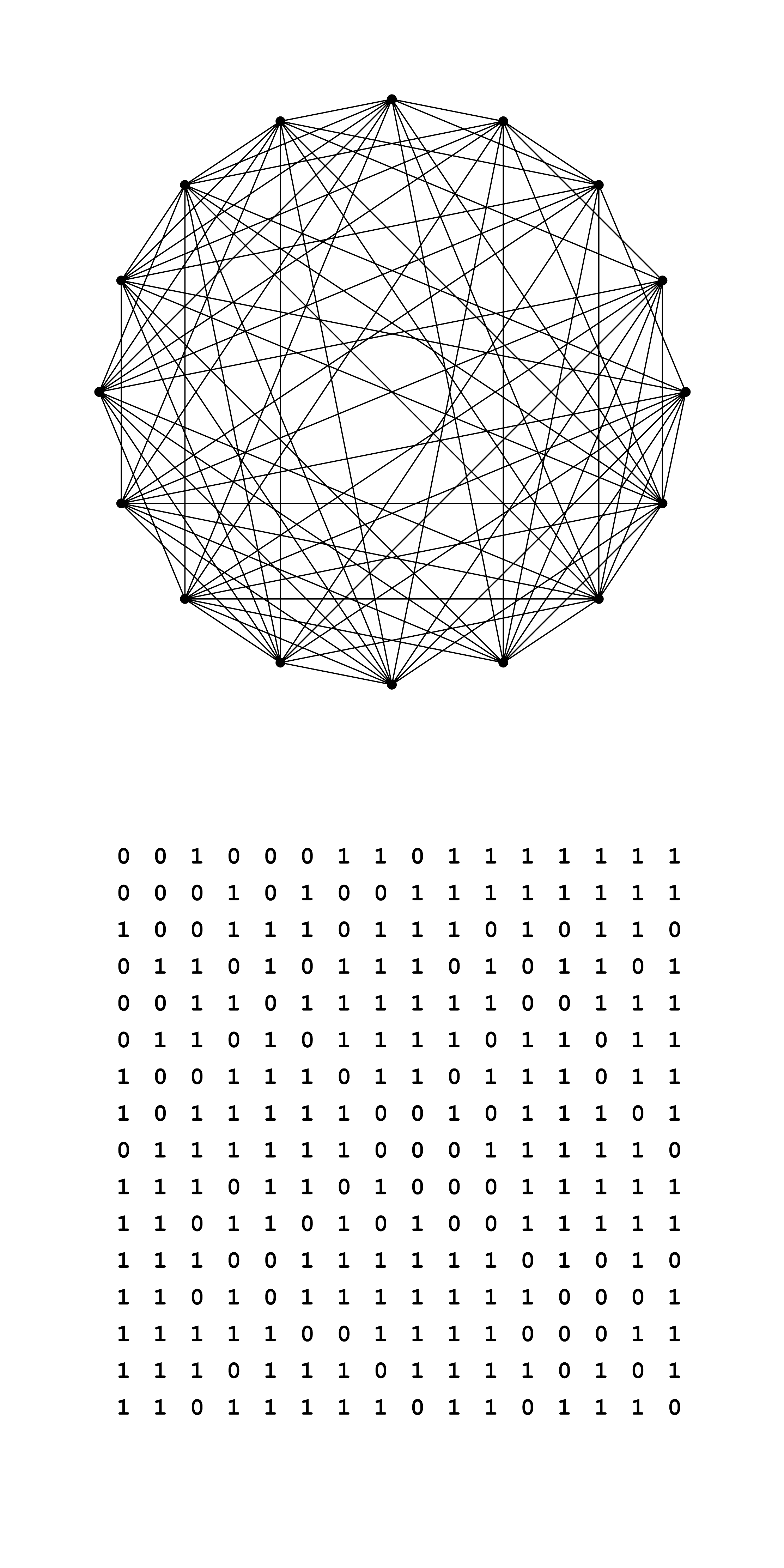}
		\caption*{\emph{$G_{78}$}}
		\label{figure: G_78}
	\end{subfigure}
	
	\begin{subfigure}[b]{0.5\textwidth}
		\centering
		\includegraphics[height=240px,width=120px]{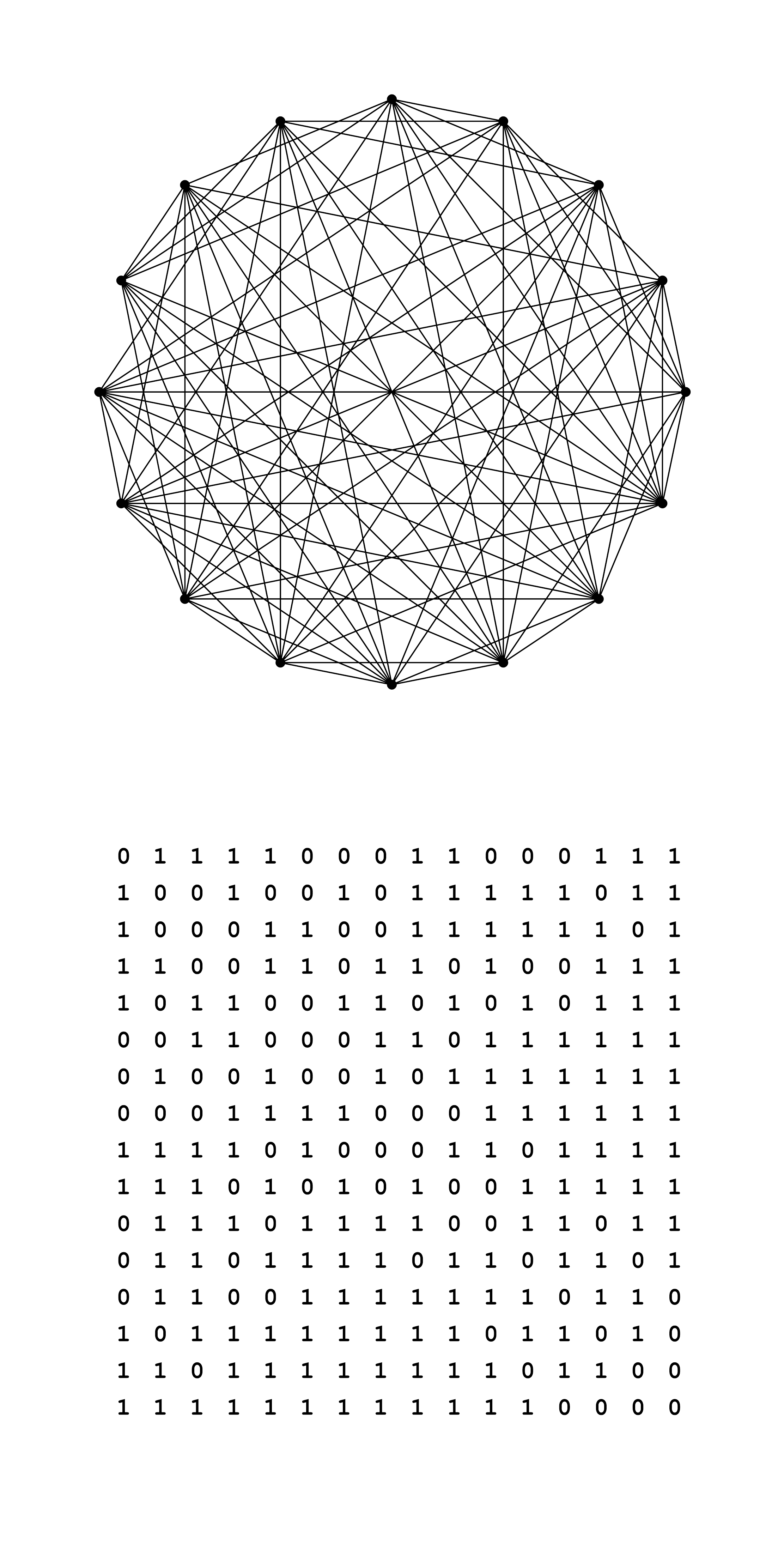}
		\caption*{\emph{$G_{134}$}}
		\label{figure: G_134}
	\end{subfigure}%
	\begin{subfigure}[b]{0.5\textwidth}
		\centering
		\includegraphics[height=240px,width=120px]{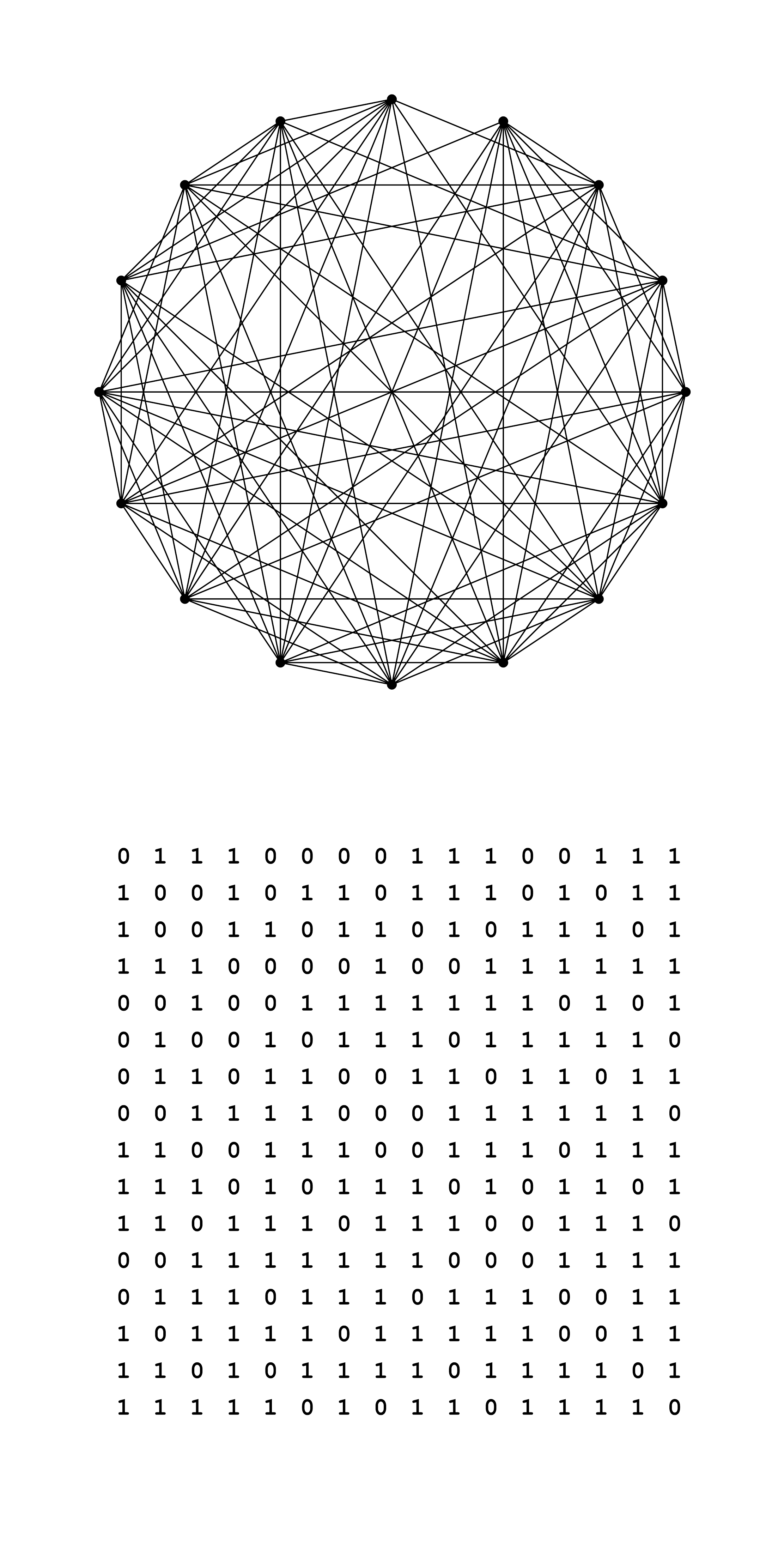}
		\caption*{\emph{$G_{135}$}}
		\label{figure: G_135}
	\end{subfigure}
	\caption{All 4 bicritical graphs in $\mathcal{H}(2, 2, 5; 6; 16)$}
	\label{figure: H(2, 2, 5; 6; 16) bicritical}
\end{figure}

Results of computations:

Step 1: Among all the graphs in $\mathcal{H}(2, 5; 6; 13)$ exactly 2 265 005 have the property that the addition of a new edge forms a new 5-clique. By adding three independent vertices to them we obtain 32 maximal graphs $\mathcal{H}(2, 2, 5; 6; 16)$.

Step 2: The number of graphs in $\mathcal{R}(3, 6; 16)$ is 2576 \cite{McK_r36_16}. Among them 5 are maximal graphs in $\mathcal{H}(2, 2, 5; 6; 16)$.

Step 3: The union of the graphs from steps 1. and 2. gives all 37 maximal graphs in $\mathcal{H}(2, 2, 5; 6; 16)$. By removing edges from them we obtain all 147 graphs in $\mathcal{H}(2, 2, 5; 6; 16)$.

We denote by $G_1, ..., G_{147}$ the graphs in $\mathcal{H}(2, 2, 5; 6; 16)$. The indexes correspond to the defined order in the \emph{nauty} programs. In Table \ref{table: H(2, 2, 5; 6; 16) properties} are listed some properties of the graphs in $\mathcal{H}(2, 2, 5; 6; 16)$. Among them there are 37 maximal, 41 minimal and 4 bicritical graphs (see Figure \ref{figure: H(2, 2, 5; 6; 16) bicritical}). The properties of the bicritical graphs are listed in Table \ref{table: H(2, 2, 5; 6; 16) bicritical properties}.

Thus, we finished the proof of Theorem \ref{theorem: abs(mathcal(H)(2, 2, 5; 6; 16)) = 147}.

All computations were done on a personal computer. The slowest part was step 1.2 of Algorithm \ref{algorithm: H(2, 2, 5; 6; 16)} which took several days to complete.

Note that to find all graphs in $\mathcal{H}(2, 2, 5; 6; 16)$ it is enough to find only these graphs from the sets $\mathcal{H}(2, 5; 6; 13)$ and $\mathcal{H}(5; 6; 10)$ for which the addition of a new edge forms a new 5-clique. In this case that does not save us much of the time needed for computer work, but later, in the proof of Corollary \ref{corollary: F_v(2_r, 5; r + 4) = r + 14}, we use that possibility.

\begin{table}[h]
	\centering
	\begin{tabular}{ | l r | l r | l r | l r | l r | l r | l r | }
		\hline
		$|\E(G)|$		& $\#$	& $\delta(G)$	& $\#$	& $\Delta(G)$	& $\#$	& $\alpha(G)$	& $\#$	& $\chi(G)$		& $\#$	& $|Aut(G)|$		& $\#$	\\
		\hline
		83				&  7		& 7			& 2			& 11		& 24		& 2			& 21		& 7			& 65		& 1			& 84		\\
		84				&  25		& 8			& 36		& 12		& 123		& 3			& 126		& 8			& 82		& 2			& 44		\\
		85				&  42		& 9			& 61		& 			& 			& 			& 			& 			& 			& 3			& 1			\\
		86				&  37		& 10		& 47		& 			& 			& 			& 			& 			& 			& 4			& 8			\\
		87				&  29		& 11		& 1			& 			& 			& 			& 			& 			& 			& 6			& 8			\\
		88				&  6		& 			& 			& 			& 			& 			& 			& 			& 			& 8			& 1			\\
		89				&  1		& 			& 			& 			& 			& 			& 			& 			& 			& 96		& 1			\\
		\hline
	\end{tabular}
	\caption{Properties of the graphs in $\mathcal{H}(2, 2, 5; 6; 16)$}
	\label{table: H(2, 2, 5; 6; 16) properties}
\end{table}

\begin{table}[h]
	\centering
	\begin{tabular}{ | l | r | r | r | r | r | r | }
		\hline
		Graph		& $|\E(G)|$		& $\delta(G)$	& $\Delta(G)$	& $\alpha(G)$	& $\chi(G)$		& $|Aut(G)|$	\\
		\hline
		$G_{74}$		&  86		& 9			& 12		& 3			& 7			& 1			\\
		$G_{78}$		&  87		& 10		& 12		& 3			& 7			& 2			\\
		$G_{134}$		&  85		& 9			& 12		& 3			& 7			& 2			\\
		$G_{135}$		&  85		& 9			& 12		& 3			& 7			& 1			\\
		\hline
	\end{tabular}
	\caption{Properties of the bicritical graphs in $\mathcal{H}(2, 2, 5; 6; 16)$}
	\label{table: H(2, 2, 5; 6; 16) bicritical properties}
\end{table}

\section{Proof of Theorem \ref{theorem: F_v(2, 2, 5; 6) = 16} and Theorem \ref{theorem: abs(mathcal(H)(3, 5; 6; 16)) = 4}}

\begin{figure}
	\centering
	\begin{subfigure}[b]{0.5\textwidth}
		\centering
		\includegraphics[height=240px,width=120px]{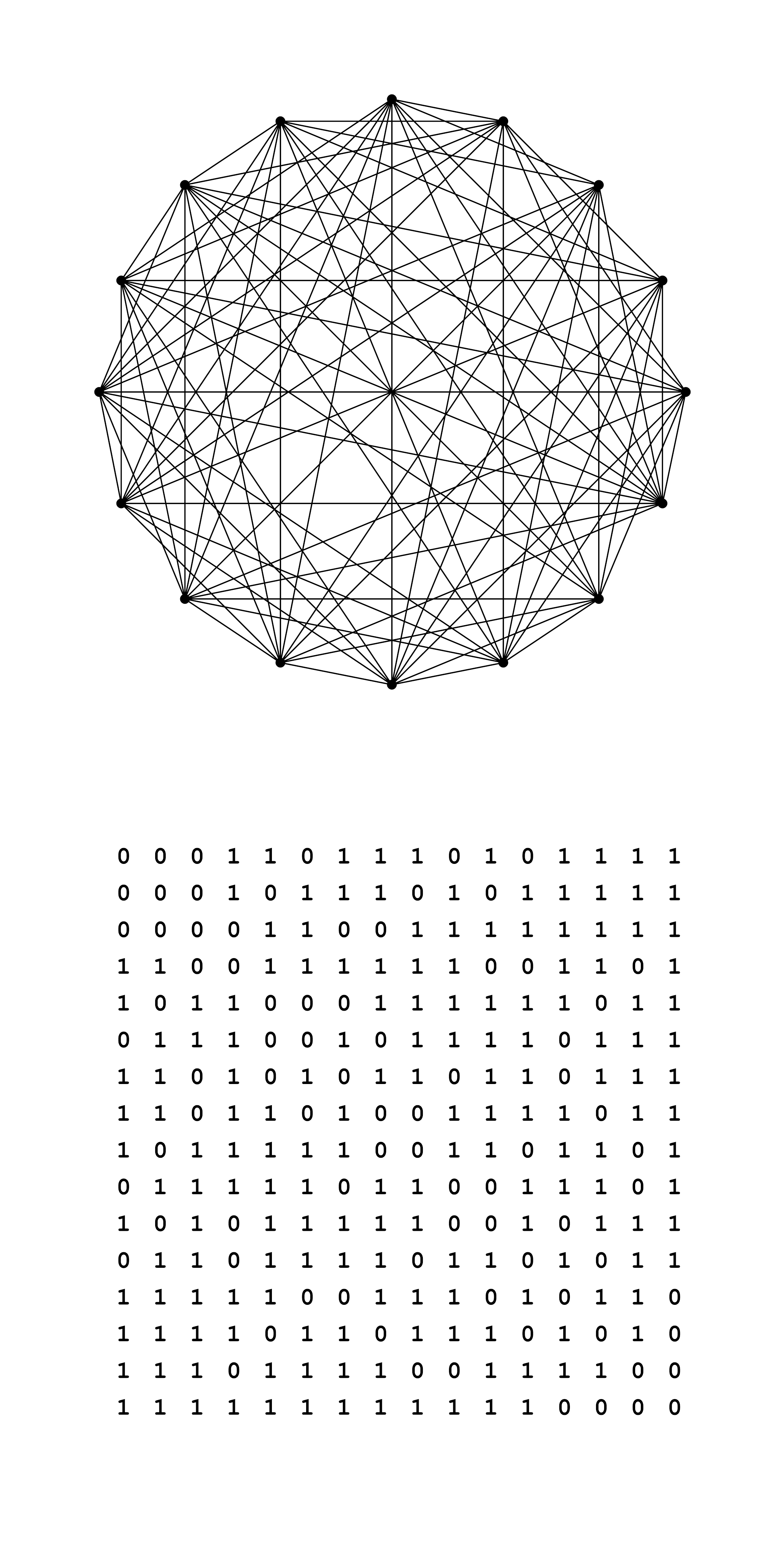}
		\caption*{\emph{$G_{50}$}}
		\label{figure: G_50}
	\end{subfigure}%
	\begin{subfigure}[b]{0.5\textwidth}
		\centering
		\includegraphics[height=240px,width=120px]{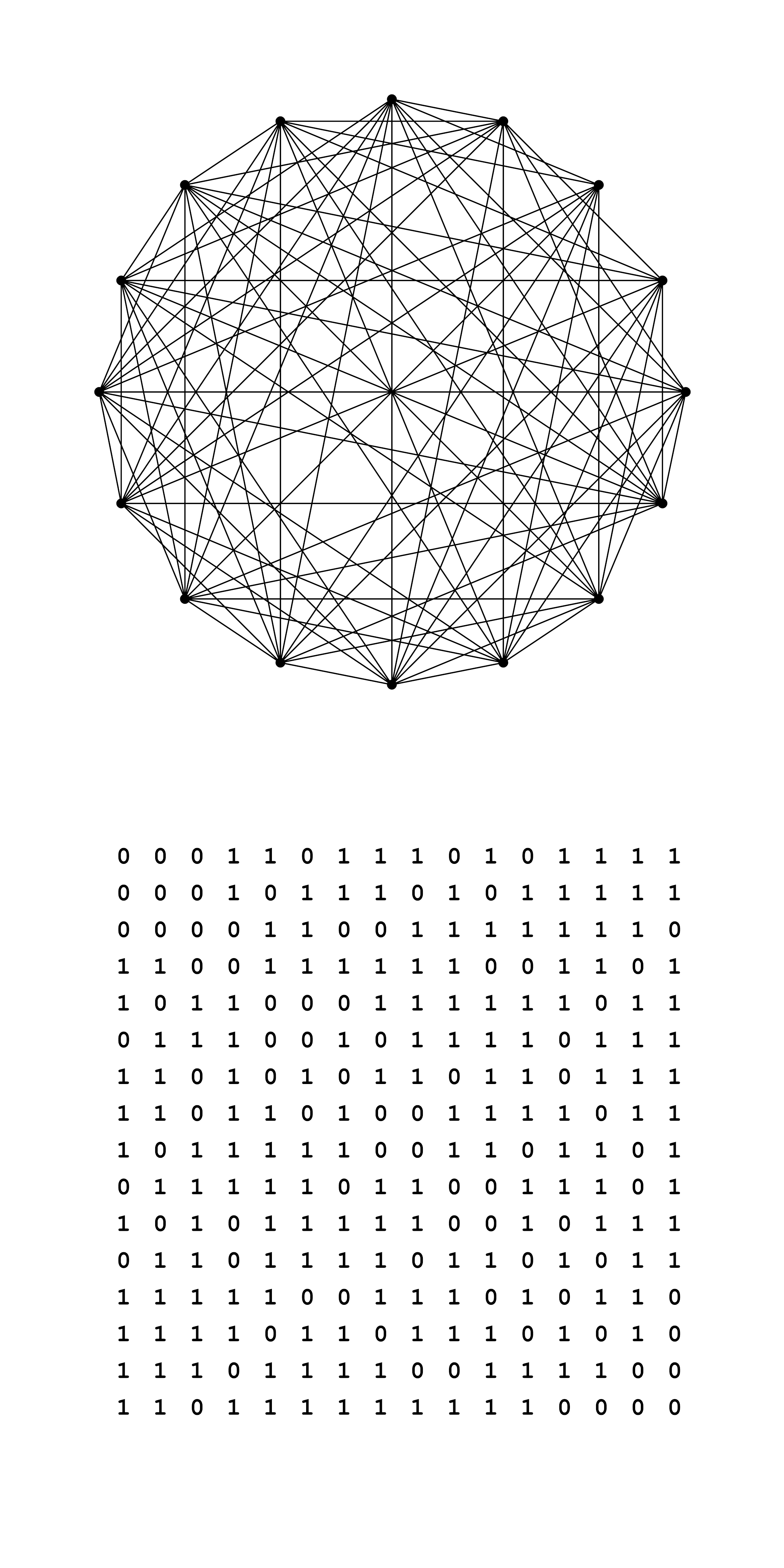}
		\caption*{\emph{$G_{51}$}}
		\label{figure: G_51}
	\end{subfigure}
	
	\begin{subfigure}[b]{0.5\textwidth}
		\centering
		\includegraphics[height=240px,width=120px]{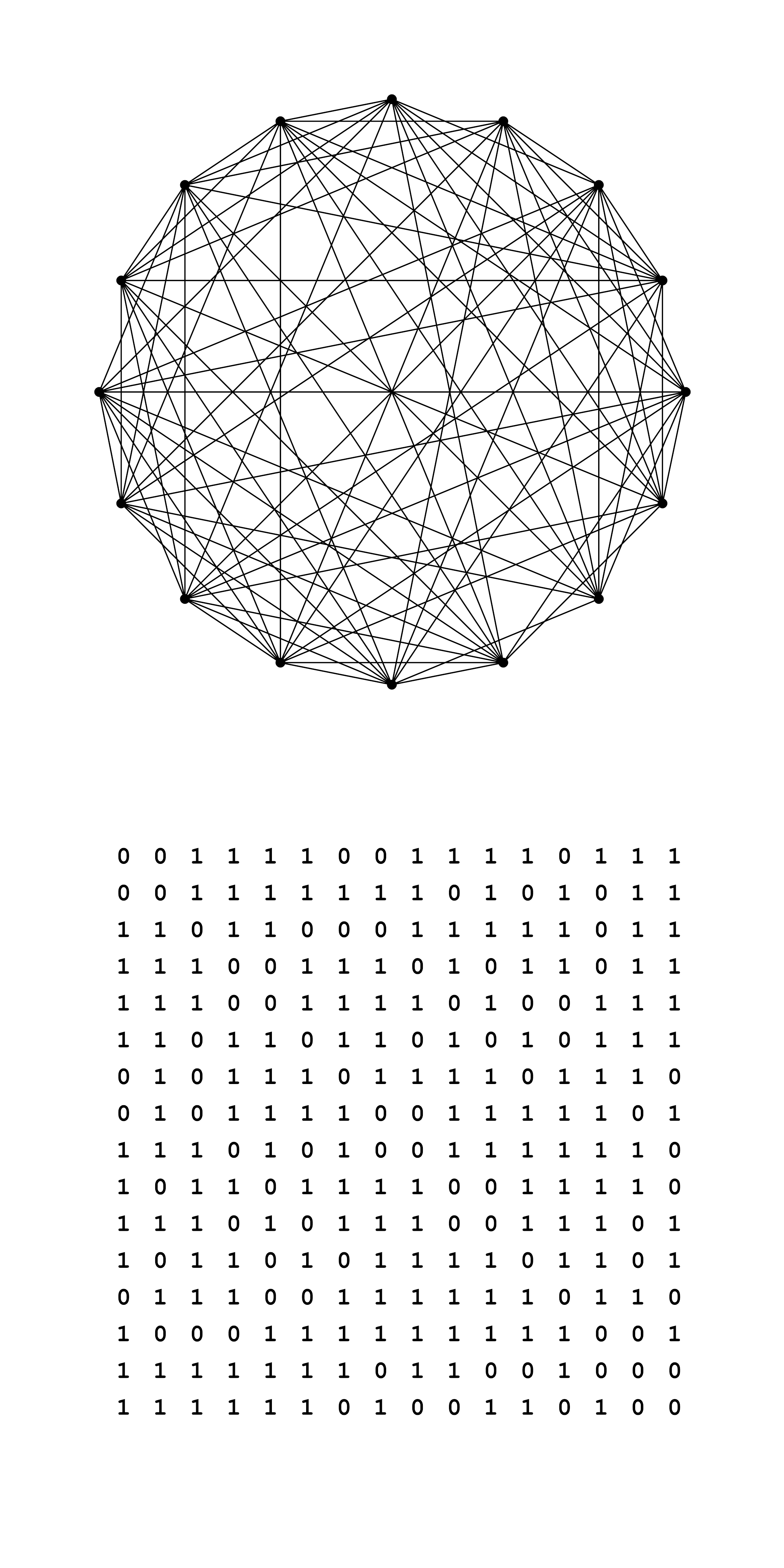}
		\caption*{\emph{$G_{81}$}}
		\label{figure: G_81}
	\end{subfigure}%
	\begin{subfigure}[b]{0.5\textwidth}
		\centering
		\includegraphics[height=240px,width=120px]{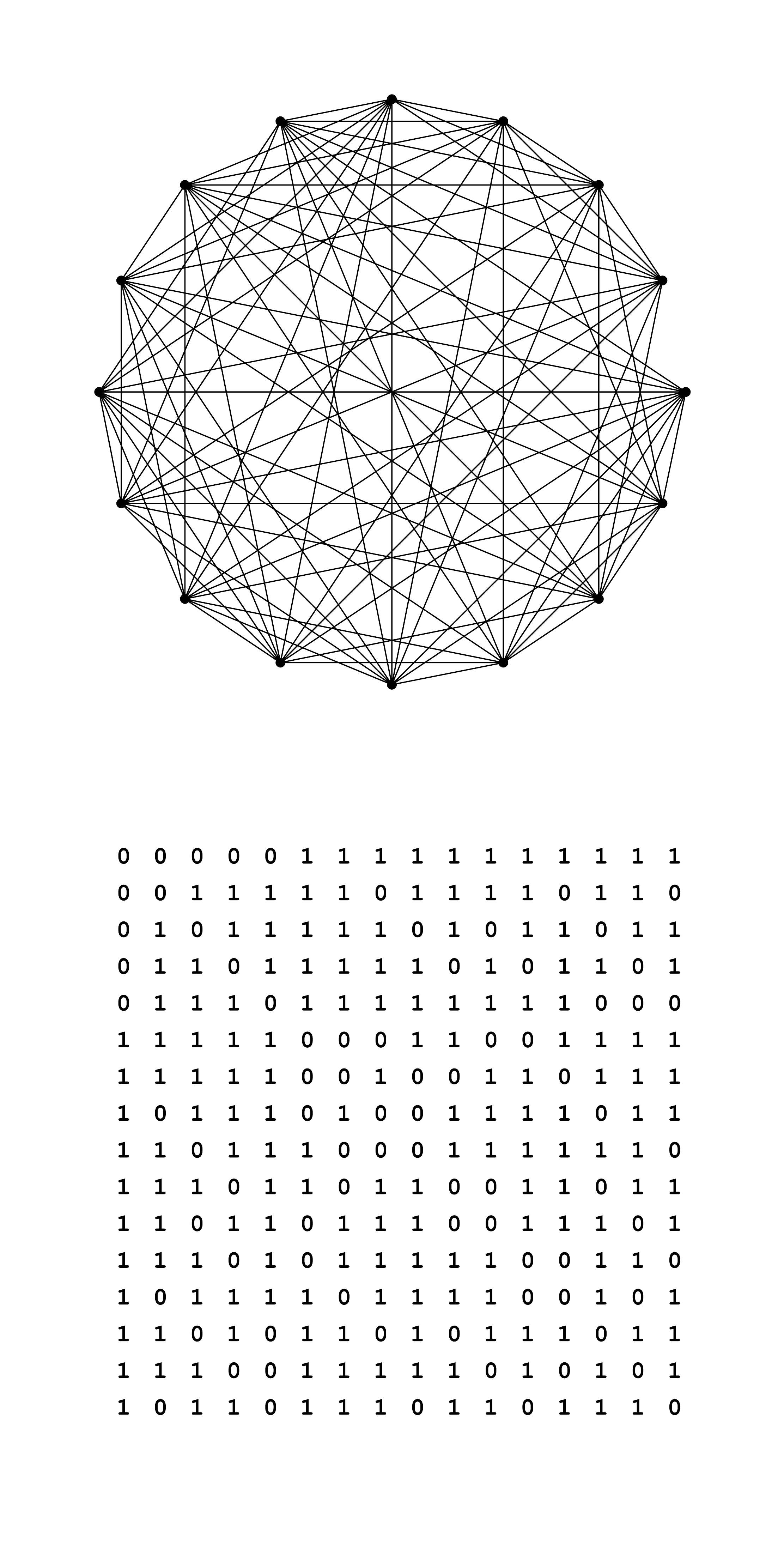}
		\caption*{\emph{$G_{146}$}}
		\label{figure: G_146}
	\end{subfigure}
	\caption{All 4 graphs in $\mathcal{H}(3, 5; 6; 16)$}
	\label{figure: H(3, 5; 6; 16)}
\end{figure}

\subsection*{Proof of Theorem \ref{theorem: F_v(2, 2, 5; 6) = 16}}

Since $\mathcal{H}(2, 2, 5; 6; 16) \neq \emptyset$, it follows that $F_v(2, 2, 5; 6) \leq 16$. With a simple algorithm, which removes a vertex and checks for inclusion in $\mathcal{H}(2, 2, 5; 6)$, we obtain $\mathcal{H}(2, 2, 5; 6; 15) = \emptyset$ which proves $F_v(2, 2, 5; 6) \geq 16$. Thus, the theorem is proved.

\begin{remark}
The lower bound $F_v(2, 2, 5; 6) \geq 16$ can be proved simpler in terms of time needed for computer work. The result $\mathcal{H}(2, 2, 5; 6; 15) = \emptyset$ can be obtained with a method similar to the one used to find all graphs in $\mathcal{H}(2, 2, 5; 6; 16)$, but in the slowest step we add 3 vertices to appropriately chosen 12-vertex graphs instead of 13-vertex graphs. A similar approach is used in the proof of the bound $F_v(3, 5; 6) \geq 16$ in \cite{SLPX12}.
\end{remark}

\subsection*{Proof of Theorem \ref{theorem: abs(mathcal(H)(3, 5; 6; 16)) = 4}}

Using that $\mathcal{H}(3, 5; 6; 16) \subseteq \mathcal{H}(2, 2, 5; 6; 16)$ by checking the graphs from Theorem \ref{theorem: abs(mathcal(H)(2, 2, 5; 6; 16)) = 147} with computer we obtain $\abs{\mathcal{H}(3, 5; 6; 16)} = 4$. The graphs from $\mathcal{H}(3, 5; 6; 16)$ are shown in Figure \ref{figure: H(3, 5; 6; 16)}. Some properties of these graphs are listed in Table \ref{table: H(3, 5; 6; 16) properties}. Thus, the theorem is proved.\\

It is interesting to note that for all these graphs the inequality (\ref{equation: G overset(v)(rightarrow) (a_1, ..., a_s) Rightarrow chi(G) geq m}) is strict. The graphs $G_{50}$ and $G_{146}$ are maximal and the other two graphs $G_{51}$ and $G_{81}$ are their subgraphs and are obtained by removing one edge. In \cite{SXP09} the inequality $F_v(3, 5; 6) \leq 16$ is proved with the help of the graph $G_{146}$. We shall note that $\abs{Aut(G_{146})} = 96$ and among all graphs in $\mathcal{H}(2, 2, 5; 6; 16)$ it has the most automorphisms.\\

\begin{table}[h]
	\centering
	\begin{tabular}{ | l | r | r | r | r | r | r | }
		\hline
		Graph		& $|\E(G)|$		& $\delta(G)$	& $\Delta(G)$	& $\alpha(G)$	& $\chi(G)$		& $|Aut(G)|$	\\
		\hline
		$G_{50}$		&  87		& 10		& 12		& 3			& 8			& 6			\\
		$G_{51}$		&  86		& 9			& 11		& 3			& 8			& 6			\\
		$G_{81}$		&  87		& 10		& 11		& 2			& 8			& 6			\\
		$G_{146}$		&  88		& 11		& 11		& 2			& 8			& 96		\\
		\hline
	\end{tabular}
	\caption{Properties of the graphs in $\mathcal{H}(3, 5; 6; 16)$}
	\label{table: H(3, 5; 6; 16) properties}
\end{table}

\section{Bounds for the numbers $F_v(a_1, ..., a_s; q)$}

First, we define a modification of the vertex Folkman numbers $F_v(a_1, ..., a_s; q)$ with the help of which we obtain upper bound for these numbers.

\begin{definition}
\label{definition: G overset(v)(rightarrow) uni(m)(p)}
Let $G$ be a graph and $m$ and $p$ be positive integer. The expression
\begin{equation*}
G \overset{v}{\rightarrow} \uni{m}{p}
\end{equation*}
means that for every choice of positive integers $a_1, ..., a_s$ ($s$ is not fixed), such that $m = \sum\limits_{i=1}^s (a_i - 1) + 1$ and $\max\set{a_1, ..., a_s} \leq p$, we have
\begin{equation*}
G \overset{v}{\rightarrow} (a_1, ..., a_s).
\end{equation*}
\end{definition}

\begin{example}
\label{example: K_m overset(v)(rightarrow) uni(m)(p)}
$K_m \overset{v}{\rightarrow} \uni{m}{p}, \quad \forall p$ (obviously).
\end{example}

\begin{example}
\label{example: overline(C)_(2p + 1) overset(v)(rightarrow) uni(p + 1)(p)}
\cite{LRU01}
Let us notice that $\overline{C}_{2p + 1} \overset{v}{\rightarrow} \uni{(p + 1)}{p}$. Indeed, let $b_1, ..., b_s$ be positive integers, such that $\sum\limits_{i = 1}^{s}(b_i - 1) + 1 = p + 1$ and $\max\set{b_1, ..., b_s} \leq p$. Assume that there exists $s$-coloring $\V(G) = V_1 \cup ... \cup V_s$, such that $V_i$ does not contain a $b_i$-clique. Then $\abs{V_i} \leq 2b_i - 2$ and $\abs{\V(G)} = \sum\limits_{i = 1}^{s}\abs{V_i} \leq 2\sum\limits_{i = 1}^{s}(b_i - 1) = 2p$ which is a contradiction.
\end{example}

Define:

$\wH{m}{p}{q} = \set{G : G \overset{v}{\rightarrow} \uni{m}{p} \mbox{ and } \omega(G) < q}$.

$\wFv{m}{p}{q} = \min\set{\abs{\V(G)} : G \in \wH{m}{p}{q}}$.

\begin{proposition}
\label{proposition: wFv(m)(p)(q) exists}
$\wH{m}{p}{q} \neq \emptyset$, i.e. $\wFv{m}{p}{q}$ exists $\Leftrightarrow$ $q > \min\set{m , p}$.
\end{proposition}

\begin{proof}
Let $\wH{m}{p}{q} \neq \emptyset$ and $G \in \wH{m}{p}{q}$. If $m \leq p$, then $G \overset{v}{\rightarrow} (m)$ and it follows $\omega(G) \geq m$. Since $\omega(G) < q$, we obtain $q > m$. Let $m > p$. Then there exist positive integers $a_1, ..., a_s$, such that $m = \sum_{i = 1}^s (a_i - 1) + 1$ and $p = \max\set{a_1, ..., a_s}$, for example $a_1 = ... = a_{m - p} = 2$ and $a_{m - p + 1} = p$. Since $G \overset{v}{\rightarrow} (a_1, ..., a_s)$, it follows that $\omega(G) \geq p$ and $q > p$. Therefore, if $\mathcal{H}(m; p; q) \neq \emptyset$, then $q > \min\set{m, p}$.

Let $q > \min\set{m, p}$. If $m \geq p$, then $q > p$. According to (\ref{equation: F_v(a_1, ..., a_s; q) exists}), for every choice of positive integers $a_1, ..., a_s$, such that $m = \sum_{i = 1}^s (a_i - 1) + 1$ and $\max\set{a_1, ..., a_s} \leq p$ there exists a graph $G(a_1, ..., a_s) \in \mathcal{H}(a_1, ..., a_s; q)$. Let $G$ be the union of all graphs $G(a_1, ..., a_s)$. It is clear that $G \in \wH{m}{p}{q}$. If $m \leq p$, then $m < q$ and therefore $K_m \in \wH{m}{p}{q}$.
\end{proof}

The following theorem gives bounds for the numbers $F_v(a_1, ..., a_s; q)$:

\begin{theorem}
\label{theorem: F_v(2_(m - p), p; q) leq F_v(a_1, ..., a_s; q) leq wFv(m)(p)(q)}
Let $a_1, ..., a_s$ be positive integers and $m$ and $p$ are defined by (\ref{equation: m and p}), $q > p$. Then
\begin{equation*}
F_v(2_{m - p}, p; q) \leq F_v(a_1, ..., a_s; q) \leq \wFv{m}{p}{q}.
\end{equation*}
\end{theorem}

\begin{proof}
The right inequality follows from the inclusion
\begin{equation*}
\wH{m}{p}{q} \subseteq \mathcal{H}(a_1, ..., a_s; q).
\end{equation*}
In order to prove the left inequality, let us notice that if $a_i \geq 3$, then
\begin{equation}
\label{equation: G overset(v)(rightarrow) (a_1, ..., a_s) Rightarrow G overset(v)(rightarrow) (a_1, ..., a_(i - 1), 2, a_i - 1, ..., a_s)}
G \overset{v}{\rightarrow} (a_1, ..., a_s) \Rightarrow G \overset{v}{\rightarrow} (a_1, ..., a_{i - 1}, 2, a_i - 1, ..., a_s).
\end{equation}
Since $m(a_1, ..., a_s) = m(a_1, ..., a_{i - 1}, 2, a_i - 1, ..., a_s)$, by successively applying (\ref{equation: G overset(v)(rightarrow) (a_1, ..., a_s) Rightarrow G overset(v)(rightarrow) (a_1, ..., a_(i - 1), 2, a_i - 1, ..., a_s)}) we obtain
\begin{equation}
\label{equation: G overset(v)(rightarrow) (a_1, ..., a_s) Rightarrow G overset(v)(rightarrow) (2_(m - p), p)}
G \overset{v}{\rightarrow} (a_1, ..., a_s) \Rightarrow G \overset{v}{\rightarrow} (2_{m - p}, p)
\end{equation}
\begin{equation}
\label{equation: G overset(v)(rightarrow) (a_1, ..., a_s) Rightarrow G overset(v)(rightarrow) (2_(m - 1))}
G \overset{v}{\rightarrow} (a_1, ..., a_s) \Rightarrow G \overset{v}{\rightarrow} (2_{m - 1}).
\end{equation}
From (\ref{equation: G overset(v)(rightarrow) (a_1, ..., a_s) Rightarrow G overset(v)(rightarrow) (2_(m - p), p)}) it follows
\begin{equation*}
F_v(a_1, ..., a_s; q) \geq F_v(2_{m - p}, p; q).
\end{equation*}
\end{proof}

Since $G \overset{v}{\rightarrow} (2_{m - 1}) \Leftrightarrow \chi(G) \geq m$, from (\ref{equation: G overset(v)(rightarrow) (a_1, ..., a_s) Rightarrow G overset(v)(rightarrow) (2_(m - 1))}) it becomes clear that
\begin{equation}
\label{equation: G overset(v)(rightarrow) (a_1, ..., a_s) Rightarrow chi(G) geq m}
G \overset{v}{\rightarrow} (a_1, ..., a_s) \Rightarrow \chi(G) \geq m, \quad\mbox{\cite{Nen01}}.
\end{equation}

This fact is used later in the proof of Theorem \ref{theorem: min_(r geq 2) set(F_v(2_r, p; r + p - 1) - r) = F_v(2_(r_0), p; r_0 + p - 1) - r_0}.

The bounds from Theorem \ref{theorem: F_v(2_(m - p), p; q) leq F_v(a_1, ..., a_s; q) leq wFv(m)(p)(q)} are useful because in general they are easier to estimate and compute than the numbers $F_v(a_1, ..., a_s)$ themselves. Later we compute the exact value of the numbers $F_v(2_{m - 5}, 5; m-1)$ (see Corollary \ref{corollary: F_v(2_r, 5; r + 4) = r + 14}) and the numbers $\wFv{m}{5}{m - 1}$ (see Theorem \ref{theorem: wFv(m)(5)(m - 1) = ...}). This way, with the help of Theorem \ref{theorem: F_v(2_(m - p), p; q) leq F_v(a_1, ..., a_s; q) leq wFv(m)(p)(q)}, Theorem \ref{theorem: F_v(a_1, ..., a_s; m - 1) = m + 9, max set(a_1, ..., a_s) = 5} is proved. Similarly, we obtain the bounds of Theorem \ref{theorem: m + 9 leq F_v(a_1, ..., a_s; m - 1) leq m + 10, max set(a_1, ..., a_s) = 6}

\begin{remark}
It is easy to see that if $q > m$, then $F_v(a_1, ..., a_s; q) = \wFv{m}{p}{q} = m$. From Theorem \ref{theorem: F_v(a_1, ..., a_s; m) = m + p} it follows $F_v(a_1, ..., a_s; m) = \wFv{m}{p}{q} = m + p$. If $q = m - 1$ and $p \leq 4$ according to (\ref{equation: F_v(a_1, ..., a_s, m - 1) = ...}), we also have $F_v(a_1, ..., a_s; q) = \wFv{m}{p}{q}$. The first case in which the upper bound in \ref{theorem: F_v(2_(m - p), p; q) leq F_v(a_1, ..., a_s; q) leq wFv(m)(p)(q)} is not reached is $m = 7, p = 5, q = 6$, since $\wFv{7}{5}{6} = 17$ (see Theorem \ref{theorem: wFv(m)(5)(m - 1) = ...}) and the corresponding numbers $F_v(a_1, ..., a_s; q) \leq 16$.
\end{remark}

\section{Some necessary results for the numbers\\ $F_v(2_r, p, r + p - 1), \ p \geq 2$}

In this section we prove that the computation of the lower bound in Theorem \ref{theorem: F_v(2_(m - p), p; q) leq F_v(a_1, ..., a_s; q) leq wFv(m)(p)(q)} in the case $q = m - 1$, i.e. computation of the numbers $F_v(2_r, p, r + p - 1)$ where $p$ is fixed, is reduced to the computation of a finite number of these numbers (Theorem \ref{theorem: min_(r geq 2) set(F_v(2_r, p; r + p - 1) - r) = F_v(2_(r_0), p; r_0 + p - 1) - r_0}).\\
It is easy to prove that
\begin{equation*}
G \overset{v}{\rightarrow} (a_1, ..., a_s) \Rightarrow K_1 + G \overset{v}{\rightarrow} (2, a_1, ..., a_s).
\end{equation*}

Therefore, it is true that
\begin{equation}
\label{equation: G overset(v)(rightarrow) (a_1, ..., a_s) Rightarrow K_t + G overset(v)(rightarrow) (2_t, a_1, ..., a_s)}
G \overset{v}{\rightarrow} (a_1, ..., a_s) \Rightarrow K_t + G \overset{v}{\rightarrow} (2_t, a_1, ..., a_s).
\end{equation}

\begin{lemma}
\label{lemma: F_v(2_r, p; r + p - 1) leq F_v(2_s, p; s + p - 1) + r - s}
Let $2 \leq s \leq r$. Then
\begin{equation*}
F_v(2_r, p; r + p - 1) \leq F_v(2_s, p; s + p - 1) + r - s.
\end{equation*}
\end{lemma}

\begin{proof}
Let $G$ be an extremal graph in $\mathcal{H}(2_s, p; s + p - 1)$. Consider $G_1 = K_{r - s} + G$. According to (\ref{equation: G overset(v)(rightarrow) (a_1, ..., a_s) Rightarrow K_t + G overset(v)(rightarrow) (2_t, a_1, ..., a_s)}), $G_1 \overset{v}{\rightarrow} (2_r, p)$. Since $\omega(G_1) = r - s + \omega(G) < r + p - 1$, it follows that $G_1 \in \mathcal{H}(2_r, p; r + p - 1)$. Therefore,

$F_v(2_r, p; r + p - 1) \leq \abs{\V(G_1)} = F_v(2_s, p; s + p - 1) + r - s$.
\end{proof}

\begin{theorem}
\label{theorem: min_(r geq 2) set(F_v(2_r, p; r + p - 1) - r) = F_v(2_(r_0), p; r_0 + p - 1) - r_0}
Let $r_0(p) = r_0$ be the smallest positive integer for which
\begin{equation*}
\min_{r \geq 2} \set{F_v(2_r, p; r + p - 1) - r} = F_v(2_{r_0}, p; r_0 + p - 1) - r_0.
\end{equation*}

Then:
\begin{flalign*}
F_v(2_r, p; r + p - 1) = F(2_{r_0}, p; r_0 + p - 1) + r - r_0, \quad r \geq r_0. && \tag{a}
\end{flalign*}
\begin{flalign*}
\mbox{if $r_0 = 2$, then} && \tag{b}
\end{flalign*}

$F_v(2_r, p; r + p - 1) = F_v(2, 2, p; p + 1) + r - 2, \quad r \geq 2$

\begin{flalign*}
\mbox{if $r_0 > 2$ and $G$ is an extremal graph in $\mathcal{H} (2_{r_0}, p; r_0 + p - 1)$, then} && \tag{c}
\end{flalign*}

$G \overset{v}{\rightarrow} (2, r_0 + p - 2).$

\begin{flalign*}
r_0 < F_v(2, 2, p; p + 1) - 2p. \mbox{ In particular, for $p = 5$ we have $r_0(5) \leq 5$} . && \tag{d}
\end{flalign*}
\end{theorem}

\begin{proof}
(a) According to the definition of $r_0 = r_0(p)$, if $r \geq 2$, then

$F_v(2_r, p; r + p - 1) - r \geq F_v(2_{r_0}, p; r_0 + p - 1) - r_0$,\\
i.e.

$F_v(2_r, p; r + p - 1) \geq F_v(2_{r_0}, p; r_0 + p - 1) + r - r_0$\\
If $r \geq r_0$, according to Lemma \ref{lemma: F_v(2_r, p; r + p - 1) leq F_v(2_s, p; s + p - 1) + r - s}, the opposite inequality is also true.

(b) This equality is the special case $r_0 = 2$ of the equality (a).

(c) Suppose the opposite is true and let $G$ be an extremal graph in $\mathcal{H}(2_{r_0}, p; r_0 + p - 1)$ and $V(G) = V_1 \cup V_2, \ V_1 \cap V_2 = \emptyset$, where $V_1$ is an independent set and $V_2$ does not contain an $(r_0 + p - 2)$-clique. We can suppose that $V_1 \neq \emptyset$. Let $G_1 = \G[V_2]$. Then $\omega(G_1) < r + p - 2$ and from $G \overset{v}{\rightarrow} (2_{r_0}, p)$ it follows $G_1 \overset{v}{\rightarrow} (2_{r_0 - 1}, p)$. Therefore, $G_1 \in \mathcal{H}(2_{r_0 - 1}, p; r_0 + p - 2)$ and

$\abs{\V(G_1)} \geq F_v(2_{r_0 - 1}, p; r_0 + p - 2)$.\\
Since $\abs{\V(G)} = F_v(2_{r_0}, p; r_0 + p - 1)$ and $\abs{\V(G_1)} \leq \abs{\V(G)} - 1$ we obtain

$F_v(2_{r_0 - 1}, p; r_0 + p - 2) - (r_0 - 1) \leq F_v(2_{r_0}, p; r_0 + p - 1) - r_0$,\\
which contradicts the definition of $r_0 = r_0(p)$.

(d) According to (\ref{equation: m + p + 2 leq F_v(a_1, ..., a_s; m - 1) leq m + 3p}), $F_v(2, 2, p; p + 1) \geq 2p + 4$. Therefore, if $r_0 = 2$, the inequality (d) is obvious.

Let $r_0 \geq 3$ and $G$ be an extremal graph in $\mathcal{H}(2_{r_0}, p; r_0 + p - 1)$. According to (c) and Theorem \ref{theorem: F_v(a_1, ..., a_s; m) = m + p}, $\abs{\V(G)} \geq 2r_0 + 2p - 3$. Let us notice that $\chi(\overline{C}_{2r_0 + 2p - 3}) = r_0 + p - 1$ and $\chi(G) \geq r_0 + p = m$ by (\ref{equation: G overset(v)(rightarrow) (a_1, ..., a_s) Rightarrow chi(G) geq m}). Therefore, $G \neq \overline{C}_{2r_0 + 2p - 3}$ and from Theorem \ref{theorem: F_v(a_1, ..., a_s; m) = m + p} we obtain

$\abs{\V(G)} = F_v(2_{r_0}, p; r_0 + p - 1) \geq 2r_0 + 2p - 2$.\\
Since $r_0 \geq 3$ from the definition of $r_0$ we have

$F_v(2_{r_0}, p; r_0 + p - 1) < F_v(2, 2, p; p + 1) + r_0 - 2$.\\
Thus, we proved that

$2r_0 + 2p - 2 < F_v(2, 2, p; p + 1) + r_0 - 2$,
i.e.

$r_0 < F_v(2, 2, p; p + 1) - 2p$.

\end{proof}

\begin{remark}
Since we suppose that $r \geq 2$, according to (\ref{equation: F_v(a_1, ..., a_s; q) exists}) all Folkman numbers in the proof of Theorem \ref{theorem: min_(r geq 2) set(F_v(2_r, p; r + p - 1) - r) = F_v(2_(r_0), p; r_0 + p - 1) - r_0} exist.
\end{remark}

\begin{example}
\label{example: r_0(2) = 4, r_0(3) = 3, r_0(4) = 2}
From (\ref{equation: F_v(a_1, ..., a_s, m - 1) = ...}) and $F_v(2, 2, 2; 3) = F_v(2, 2, 2, 2; 4) = 11$ it follows $r_0(2) = 4$, and from (\ref{equation: F_v(a_1, ..., a_s, m - 1) = ...}) and $F_v(2, 2, 3; 4) = 14$ it follows $r_0(3) = 3$. Also, from (\ref{equation: F_v(a_1, ..., a_s, m - 1) = ...}) we see that $r_0(4) = 2$.
\end{example}

We suppose that the following is true:
\begin{conjecture}
\label{conjecture: F_v(2_r, p; r + p - 1) = F_v(2, 2, p; p + 1) + r - 2}
If $p \geq 4$, then $r_0(p) = 2$ and therefore, according to Theorem \ref{theorem: min_(r geq 2) set(F_v(2_r, p; r + p - 1) - r) = F_v(2_(r_0), p; r_0 + p - 1) - r_0}(b),
\begin{equation*}
F_v(2_r, p; r + p - 1) = F_v(2, 2, p; p + 1) + r - 2, \quad r \geq 2.
\end{equation*}
\end{conjecture}

In this paper we prove this conjecture in the case $p = 5$ (see Theorem \ref{theorem: r_0(5) = 2} and Corollary \ref{corollary: F_v(2_r, 5; r + 4) = r + 14}).

\begin{corollary}
\label{corollary: F_v(a_1, ... a_s; m - 1) geq F_v(2_(r_0), p; r_0 + p - 1) + r - r_0}
Let $a_1, ..., a_s$ be positive integers, $m$ and $p$ are defined by (\ref{equation: m and p}), $m \geq p + 2$ and $r = m - p \geq r_0(p)$. Then
\begin{equation*}
F_v(a_1, ... a_s; m - 1) \geq F_v(2_{r_0}, p; r_0 + p - 1) + r - r_0.
\end{equation*}
In particular, if $r_0 = 2$, then
\begin{equation*}
F_v(a_1, ..., a_s; m - 1) \geq F_v(2, 2, p; p + 1) + r - 2.
\end{equation*}
\end{corollary}

\begin{proof}
According to Theorem \ref{theorem: F_v(2_(m - p), p; q) leq F_v(a_1, ..., a_s; q) leq wFv(m)(p)(q)},

$F_v(a_1, ..., a_s; m - 1) \geq F_v(2_r, p; r + p - 1)$.\\
From this inequality and Theorem \ref{theorem: min_(r geq 2) set(F_v(2_r, p; r + p - 1) - r) = F_v(2_(r_0), p; r_0 + p - 1) - r_0}(a) we obtain the desired inequality.
\end{proof}

\section{Computation of $r_0(5)$}

In this section we prove the following:

\begin{theorem}
\label{theorem: r_0(5) = 2}
$r_0(5) = 2$
\end{theorem}

\begin{proof}
From Theorem \ref{theorem: min_(r geq 2) set(F_v(2_r, p; r + p - 1) - r) = F_v(2_(r_0), p; r_0 + p - 1) - r_0}(d) we have $r_0(5) \leq 5$. Therefore, we have to prove that $r_0(5) \neq 3$, $r_0(5) \neq 4$ and $r_0(5) \neq 5$, i.e. we have to prove the inequalities $F_v(2, 2, 2, 5; 7) > 16$, $F_v(2, 2, 2, 2, 5; 8) > 17$, $F_v(2, 2, 2, 2, 2, 5; 9) > 18$.

The proof of each of these  three inequalities consists of several steps, similarly to the proof of Theorem \ref{theorem: abs(mathcal(H)(2, 2, 5; 6; 16)) = 147}. Since not all graphs in $\mathcal{R}(3, 7)$ are known to us, this time in the process of extending graphs to maximal ones we are adding two independent vertices instead of three:

\begin{algorithm}
\label{algorithm: H(2_r, 5; q; n)}
Finding all maximal graphs in $\mathcal{H}(2_r, 5; q; n)$ starting from all maximal graphs in $\mathcal{H}(2_{r - 1}, 5; q; n - 2)$.

1. By removing edges from the maximal graphs in $\mathcal{H}(2_{r - 1}, 5; q; n - 2)$ find all graphs in this set which have the property that the addition of a new edge forms a new $(q - 1)$-clique.

2. Using Algorithm \ref{algorithm: add_is_kn_free} add two independent vertices to the graphs from step 1. to obtain all maximal graphs in $\mathcal{H}(2_r, 5; q; n)$.
\end{algorithm}

\subsection{Proof of $F_v(2, 2, 2, 5; 7) > 16$}

By checking all 10-vertex graphs we find the maximal graphs in $\mathcal{H}(5; 7; 10)$. Starting from them, by successively applying Algorithm \ref{algorithm: H(2_r, 5; q; n)}($n = 12, 14, 16; q = 7; r = 1, 2, 3$) we obtain the maximal graphs in the sets $\mathcal{H}(2, 5; 7; 12)$, $\mathcal{H}(2, 2, 5; 7; 14)$ and $\mathcal{H}(2, 2, 2, 5; 7; 16)$. The results are described in Table \ref{table: F_v(2, 2, 2, 5; 7) > 16}. There we can see that $\mathcal{H}(2, 2, 2, 5; 7; 16) = \emptyset$ and therefore $F_v(2, 2, 2, 5; 7) > 16$.

\begin{table}[h]
	\centering
	\begin{tabular}{| p{3.5cm} | p{2.5cm} | p{5cm} |}
		\hline
		set											& maximal graphs	& graphs in which the addition of an edge forms a new 6-clique	\\
		\hline
		$\mathcal{H}(5; 7; 10)$						& 8					& 324		\\
		$\mathcal{H}(2, 5; 7; 12)$					& 56				& 104 283	\\
		$\mathcal{H}(2, 2, 5; 7; 14)$				& 420				& 2 614 547	\\
		$\mathcal{H}(2, 2, 2, 5; 7; 16)$			& 0					& 0			\\
		\hline
	\end{tabular}
	\caption{Steps in the proof of $F_v(2, 2, 2, 5; 7) > 16$}
	\label{table: F_v(2, 2, 2, 5; 7) > 16}
\end{table}

\subsection{Proof of $F_v(2, 2, 2, 2, 5; 8) > 17$}

By checking all 9-vertex graphs we find the maximal graphs in $\mathcal{H}(5; 8; 9)$. Starting from them, by successively applying Algorithm \ref{algorithm: H(2_r, 5; q; n)}($n = 11, 13, 15, 17; q = 8; r = 1, 2, 3, 4$) we obtain the maximal graphs in the sets $\mathcal{H}(2, 5; 8; 11)$, $\mathcal{H}(2, 2, 5; 8; 13)$, $\mathcal{H}(2, 2, 2, 5; 8; 15)$ and $\mathcal{H}(2, 2, 2, 2, 5; 8; 17)$. The results are described in Table \ref{table: F_v(2, 2, 2, 2, 5; 8) > 17}. There we can see that $\mathcal{H}(2, 2, 2, 2, 5; 8; 17) = \emptyset$ and therefore $F_v(2, 2, 2, 2, 5; 8) > 17$.

\begin{table}[h]
	\centering
	\begin{tabular}{| p{3.5cm} | p{2.5cm} | p{5cm} |}
		\hline
		set										& maximal graphs	& graphs in which the addition of an edge forms a new 7-clique	\\
		\hline
		$\mathcal{H}(5; 8; 9)$					& 2					& 13		\\
		$\mathcal{H}(2, 5; 8; 11)$				& 8					& 326		\\
		$\mathcal{H}(2, 2, 5; 8; 13)$			& 56				& 105 138	\\
		$\mathcal{H}(2, 2, 2, 5; 8; 15)$		& 423				& 2 616 723	\\
		$\mathcal{H}(2, 2, 2, 2, 5; 8; 17)$		& 0					& 0			\\
		\hline
	\end{tabular}
	\caption{Steps in the proof of $F_v(2, 2, 2, 2, 5; 8) > 17$}
	\label{table: F_v(2, 2, 2, 2, 5; 8) > 17}
\end{table}

\subsection{Proof of $F_v(2, 2, 2, 2, 2, 5; 9) > 18$}

By checking all 10-vertex graphs we find the maximal graphs in $\mathcal{H}(2, 5; 9; 10)$. Starting from them, by successively applying Algorithm \ref{algorithm: H(2_r, 5; q; n)}($n = 12, 14, 16, 18; q = 9; r = 2, 3, 4, 5$) we obtain the maximal graphs in the sets $\mathcal{H}(2, 2, 5; 9; 12)$, $\mathcal{H}(2, 2, 2, 5; 9; 14)$, $\mathcal{H}(2, 2, 2, 2, 5; 9; 16)$ and $\mathcal{H}(2, 2, 2, 2, 2, 5; 9; 18)$. The results are described in Table \ref{table: F_v(2, 2, 2, 2, 2, 5; 9) > 18}. There we can see that $\mathcal{H}(2, 2, 2, 2, 2, 5; 9; 18) = \emptyset$  and therefore $F_v(2, 2, 2, 2, 2, 5; 9) > 18$.

\begin{table}[h]
	\centering
	\begin{tabular}{| p{3.5cm} | p{2.5cm} | p{5cm} |}
		\hline
		set										& maximal graphs	& graphs in which the addition of an edge forms a new 8-clique	\\
		\hline
		$\mathcal{H}(2, 5; 9; 10)$				& 2					& 13		\\
		$\mathcal{H}(2, 2, 5; 9; 12)$			& 8					& 327		\\
		$\mathcal{H}(2, 2, 2, 5; 9; 14)$		& 56				& 105 294	\\
		$\mathcal{H}(2, 2, 2, 2, 5; 9; 16)$		& 423				& 2 616 741	\\
		$\mathcal{H}(2, 2, 2, 2, 2, 5; 9; 18)$	& 0					& 0			\\
		\hline
	\end{tabular}
	\caption{Steps in the proof of $F_v(2, 2, 2, 2, 2, 5; 9) > 18$}
	\label{table: F_v(2, 2, 2, 2, 2, 5; 9) > 18}
\end{table}

Thus, the proof of Theorem \ref{theorem: r_0(5) = 2} is finished.
\end{proof}

All computations were done on a personal computer. The slowest part was the proof of $F_v(2, 2, 2, 2, 2, 5; 9) > 18$ which took several days to complete.

From Theorem \ref{theorem: r_0(5) = 2} and Theorem \ref{theorem: min_(r geq 2) set(F_v(2_r, p; r + p - 1) - r) = F_v(2_(r_0), p; r_0 + p - 1) - r_0}(b) we obtain:

\begin{corollary}
\label{corollary: F_v(2_r, 5; r + 4) = r + 14}
$F_v(2_r, 5; r + 4) = r + 14, \quad r \geq 2$.
\end{corollary}

\section{Computation of the numbers $\wFv{m}{5}{m - 1}$ and proof of Theorem \ref{theorem: F_v(a_1, ..., a_s; m - 1) = m + 9, max set(a_1, ..., a_s) = 5}}

Let us remind that $\wH{m}{p}{q}$ and $\wFv{m}{p}{q}$ are defined in Section 4.\\

We need the following

\begin{lemma}
\label{lemma: K_(m - m_0) + G overset(v)(rightarrow) uni(m)(p)}
\cite{Nen02}
Let $m_0$ and $p$ be positive integers and $G \overset{v}{\rightarrow} \uni{m_0}{p}$. Then for every positive integer $m \geq m_0$ it is true that $K_{m - m_0} + G \overset{v}{\rightarrow} \uni{m}{p}$.
\end{lemma}

This lemma is formulated in an obviously equivalent way and is proved by induction with respect to $m \geq m_0$ in \cite{Nen02} as Lemma 3.

\begin{theorem}
\label{theorem: wFv(m)(p)(m - m_0 + q) leq wFv(m_0)(p)(q) + m - m_0}
Let $m$, $m_0$, $p$ and $q$ be positive integers, $m \geq m_0$ and $q > \min\set{m_0, p}$. Then
\begin{equation*}
\wFv{m}{p}{m - m_0 + q} \leq \wFv{m_0}{p}{q} + m - m_0.
\end{equation*}
\end{theorem}

\begin{proof}
Let $G_0 \in \wH{m_0}{p}{q}$, $\abs{V(G_0)} = \wFv{m_0}{p}{q}$ and $G = K_{m - m_0} + G_0$. According to Lemma \ref{lemma: K_(m - m_0) + G overset(v)(rightarrow) uni(m)(p)}, $G \overset{v}{\rightarrow} \uni{m}{p}$. Since $\omega(G) = m - m_0 + \omega(G_0) < m - m_0 + q$, it follows that $G \in \wH{m}{p}{m - m_0 + q}$. Therefore, $\wFv{m}{p}{m - m_0 + q} \leq \abs{\V(G)} = \wFv{m_0}{p}{q} + m - m_0$.
\end{proof}

The following obvious proposition will be used in the proof of Theorem \ref{theorem: wFv(m)(5)(m - 1) = ...}.
\begin{proposition}
\label{proposition: G overset(v)(rightarrow) (a_1, ..., a_(i - 1), k, a_i - k + 1, a_(i + 1), ..., a_s)}
Let $a_1, ..., a_s$ be positive integers, $a_i \geq k$ and $G \overset{v}{\rightarrow} (a_1, ..., a_s)$. Then
\begin{equation*}
G \overset{v}{\rightarrow} (a_1, ..., a_{i - 1}, k, a_i - k + 1, a_{i + 1}, ..., a_s).
\end{equation*}
\end{proposition}

According to proposition \ref{proposition: wFv(m)(p)(q) exists}, we have
\begin{equation}
\label{equation: wFv(m)(5)(m - 1) exists}
\wFv{m}{5}{m - 1} \mbox{ exists } \Leftrightarrow m \geq 7.
\end{equation}
We prove the following
\begin{theorem}
\label{theorem: wFv(m)(5)(m - 1) = ...}
The following equalities are true:

\begin{equation*}
\wFv{m}{5}{m - 1} = 
\begin{cases}
17, & \emph{if $m = 7$}\\
m + 9, & \emph{if $m \geq 8$}.
\end{cases}
\end{equation*}

\end{theorem}

\begin{proof}
\emph{Case 1.} $m = 7$. According to Theorem \ref{theorem: F_v(2_(m - p), p; q) leq F_v(a_1, ..., a_s; q) leq wFv(m)(p)(q)} and Theorem \ref{theorem: F_v(2, 2, 5; 6) = 16}, $\wFv{7}{5}{6} \geq F_v(2, 2, 5; 6) = 16$. With the help of the computer we check that none of the 4 graphs in $\mathcal{H}(3, 5; 6; 16)$ (see Figure \ref{figure: H(3, 5; 6; 16)}) belongs to $\mathcal{H}(4, 4; 6; 16)$. Therefore, $\wFv{7}{5}{6} \geq 17$.

By adding one vertex to the graphs from $\mathcal{H}(2, 2, 5; 6; 16)$ we obtain 17 vertex graphs, 343 of which belong to both $\mathcal{H}(3, 5; 6; 17)$ and $\mathcal{H}(4, 4; 6; 17)$. The graph $\Gamma_1$ given on Figure \ref{figure: H(3, 5; 6; 17) cap H(4, 4; 6; 17)} is one of these graphs (it is the only one with independence number 4). We will prove that $\Gamma_1 \in \wH{7}{5}{6}$. Since $\omega(\Gamma_1) = 5$ it remains to be proved that if $2 \leq b_1 \leq ... \leq b_s \leq 5$ are positive integers, such that $\sum_{i = 1}^s (b_1 - 1) + 1 = 7$, then $\Gamma_1 \overset{v}{\rightarrow} (b_1, ..., b_s)$. The following cases are possible:

$s = 2, b_1 = 3, b_2 = 5$

$s = 2, b_1 = b_2 = 4$

$s = 3, b_1 = b_2 = 2, b_3 = 5$

$s = 3, b_1 = 2, b_2 = 3, b_3 = 4$

$s = 3, b_1 = b_2 = b_3 = 3$

$s = 4, b_1 = b_2 = b_3 = 2, b_4 = 4$

$s = 4, b_1 = b_2 = 2, b_3 = b_4 = 3$

$s = 5, b_1 = b_2 = b_3 = b_4 = 2, b_5 = 3$

$s = 6, b_1 = b_2 = b_3 = b_4 = b_5 = b_6 = 2$

By construction, $\Gamma_1 \overset{v}{\rightarrow} (3, 5)$ and $\Gamma_1 \overset{v}{\rightarrow} (4, 4)$. From Proposition \ref{proposition: G overset(v)(rightarrow) (a_1, ..., a_(i - 1), k, a_i - k + 1, a_(i + 1), ..., a_s)} and $\Gamma_1 \overset{v}{\rightarrow} (3, 5)$ it follows $\Gamma_1 \overset{v}{\rightarrow} (2, 2, 5)$, $\Gamma_1 \overset{v}{\rightarrow} (2, 3, 4)$ and $\Gamma_1 \overset{v}{\rightarrow} (3, 3, 3)$. Consequently, we have

$\Gamma_1 \overset{v}{\rightarrow} (3, 3, 3) \Rightarrow \Gamma_1 \overset{v}{\rightarrow} (2, 2, 3, 3)$

$\Gamma_1 \overset{v}{\rightarrow} (2, 2, 5) \Rightarrow \Gamma_1 \overset{v}{\rightarrow} (2, 2, 2, 4)$

$\Gamma_1 \overset{v}{\rightarrow} (2, 2, 2, 4) \Rightarrow \Gamma_1 \overset{v}{\rightarrow} (2, 2, 2, 2, 3)$

$\Gamma_1 \overset{v}{\rightarrow} (2, 2, 2, 2, 3) \Rightarrow \Gamma_1 \overset{v}{\rightarrow} (2, 2, 2, 2, 2, 2)$

We proved that $\Gamma_1 \in \wH{7}{5}{6}$. Therefore, $\wFv{7}{5}{6} \leq \abs{\V(\Gamma_1)} = 17$.

\begin{figure}
	\centering
	\begin{subfigure}[b]{0.5\textwidth}
		\centering
		\includegraphics[height=240px,width=120px]{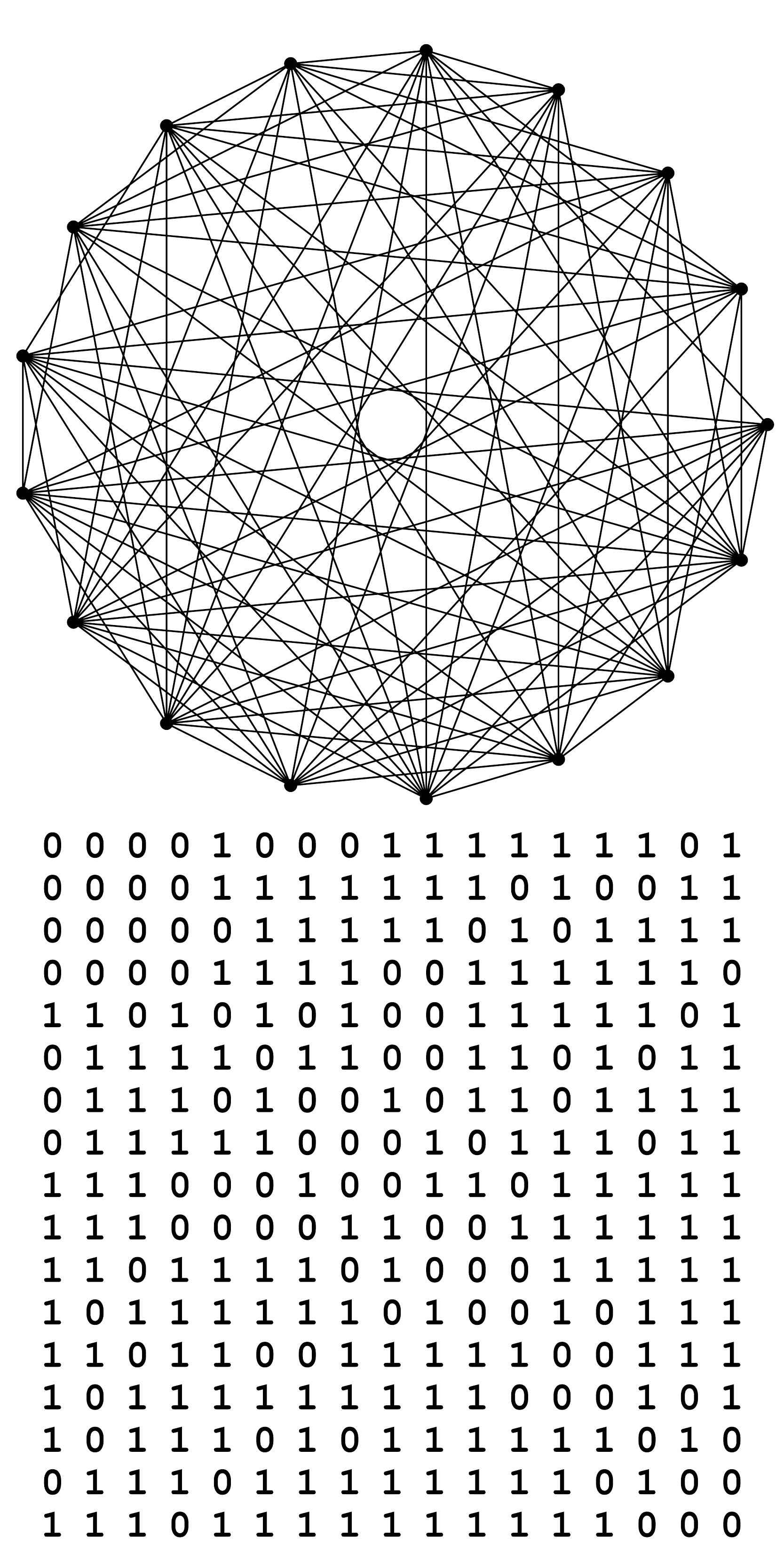}
		\caption*{\emph{$\Gamma_1$}}
		\label{figure: Gamma_1}
	\end{subfigure}
	\caption{Graph $\Gamma_1$, 17-vertex graph in $\mathcal{H}(3, 5; 6; 17) \cap \mathcal{H}(4, 4; 6; 17)$}
	\label{figure: H(3, 5; 6; 17) cap H(4, 4; 6; 17)}
\end{figure}

\emph{Case 2.} $m = 8$. According to Theorem \ref{theorem: F_v(2_(m - p), p; q) leq F_v(a_1, ..., a_s; q) leq wFv(m)(p)(q)} and Corollary \ref{corollary: F_v(2_r, 5; r + 4) = r + 14} $\wFv{8}{5}{7} \geq F_v(2, 2, 2, 5; 7) = 17$. To prove the upper bound, consider the 17-vertex graph $\Gamma_2 \in \mathcal{H}(4, 5; 7; 17)$ which is shown on Figure \ref{figure: H(4, 5; 7; 17)}. Appendix A describes the method to obtain this graph. By construction $\omega(\Gamma_2) = 6$ and $\Gamma_2 \overset{v}{\rightarrow} (4, 5)$. As in Case 1., we prove that from $\Gamma_2 \overset{v}{\rightarrow} (4, 5)$ it follows $\Gamma_2 \overset{v}{\rightarrow} \uni{8}{5}$. Therefore, $\Gamma_2 \in \wH{8}{5}{7}$ and $\wFv{8}{5}{7} \leq \abs{\V(\Gamma_2)} = 17$. 

\begin{figure}
	\centering
	\begin{subfigure}[b]{0.5\textwidth}
		\centering
		\includegraphics[height=240px,width=120px]{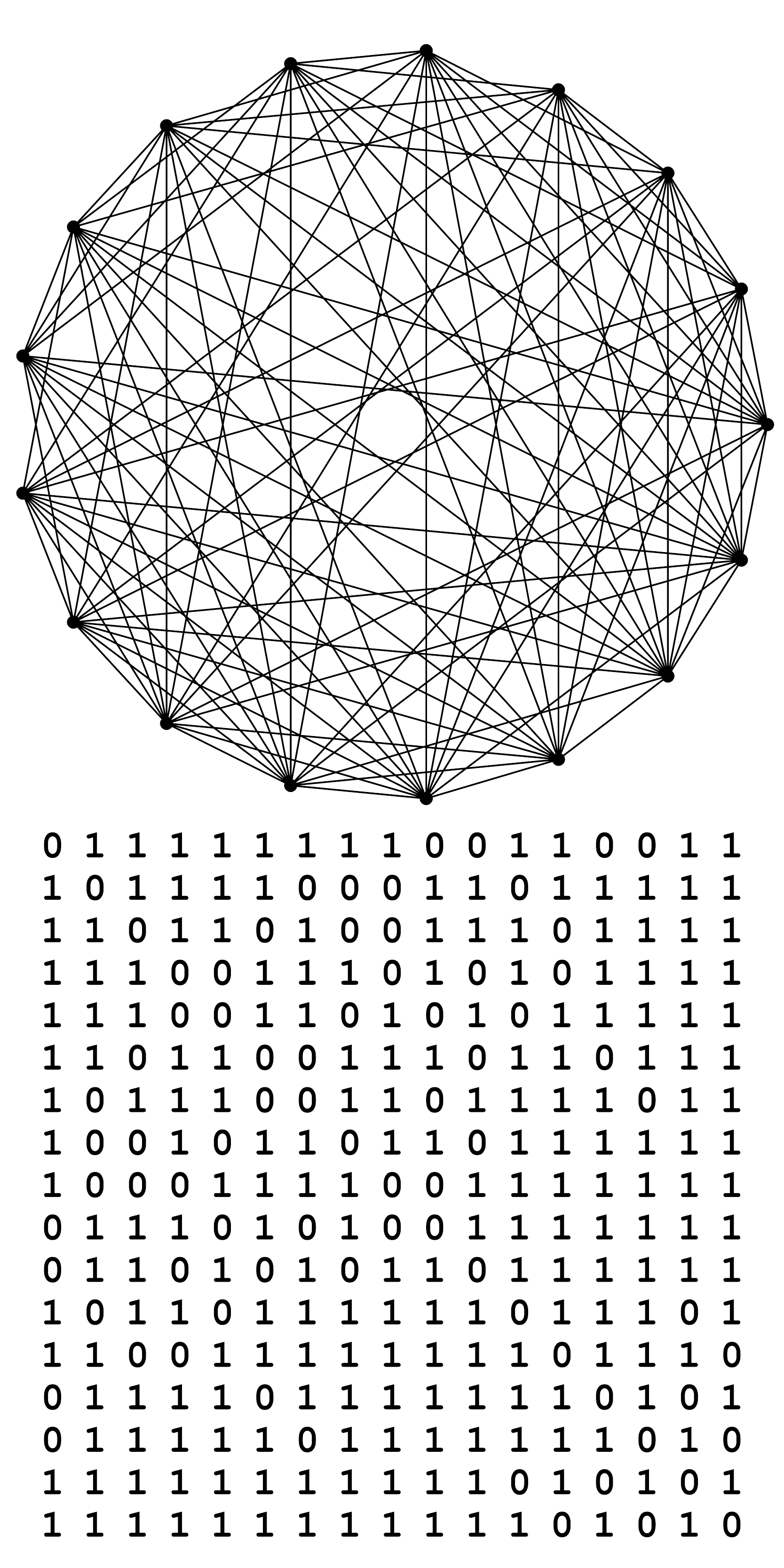}
		\caption*{\emph{$\Gamma_2$}}
		\label{figure: Gamma_2}
	\end{subfigure}
	\caption{Graph $\Gamma_2$, 17-vertex graph in $\mathcal{H}(4, 5; 7; 17)$}
	\label{figure: H(4, 5; 7; 17)}
\end{figure}

\emph{Case 3.} $m > 8$. From Theorem \ref{theorem: F_v(2_(m - p), p; q) leq F_v(a_1, ..., a_s; q) leq wFv(m)(p)(q)} and Corollary \ref{corollary: F_v(2_r, 5; r + 4) = r + 14} it follows $\wFv{m}{5}{m - 1} \geq m + 9$. From Theorem \ref{theorem: wFv(m)(p)(m - m_0 + q) leq wFv(m_0)(p)(q) + m - m_0}($m_0 = 8, p = 5, q = 7$) and $\wFv{8}{5}{7} = 17$ it follows $\wFv{m}{5}{m - 1} \leq m + 9$.
\end{proof}

\subsection*{Proof of Theorem \ref{theorem: F_v(a_1, ..., a_s; m - 1) = m + 9, max set(a_1, ..., a_s) = 5}}

Since $m \geq 7$, only the following two cases are possible:\\

\emph{Case 1.} $m = 7$. In this case $F_v(2, 2, 5; 6)$ and $F_v(3, 5; 6)$ are the only canonical vertex Folkman numbers of the form $F_v(a_1, ..., a_s; m - 1)$. The equality $F_v(2, 2, 5; 6) = 16$ is proved in this work as Theorem \ref{theorem: F_v(2, 2, 5; 6) = 16}, and the equality $F_v(3, 5; 6) = 16$ is proved in \cite{SLPX12} (see also Corollary \ref{corollary: F_v(3, 5; 6) = 16}).

\emph{Case 2.} $m \geq 8$. In this case Theorem \ref{theorem: F_v(a_1, ..., a_s; m - 1) = m + 9, max set(a_1, ..., a_s) = 5} follows easily from Theorem \ref{theorem: wFv(m)(5)(m - 1) = ...}, Theorem \ref{theorem: F_v(2_(m - p), p; q) leq F_v(a_1, ..., a_s; q) leq wFv(m)(p)(q)}($q = m - 1$) and Corollary \ref{corollary: F_v(2_r, 5; r + 4) = r + 14}.

\section{Proof of Theorem \ref{theorem: m + 9 leq F_v(a_1, ..., a_s; m - 1) leq m + 10, max set(a_1, ..., a_s) = 6}}

According to Corollary \ref{corollary: F_v(2_r, 5; r + 4) = r + 14}, $F_v(2, 2, 2, 5; 7) = 17$. From Proposition \ref{proposition: G overset(v)(rightarrow) (a_1, ..., a_(i - 1), k, a_i - k + 1, a_(i + 1), ..., a_s)} it follows $F_v(2, 2, 6; 7) \geq F_v(2, 2, 2, 5; 7)$. Therefore, $F_v(2, 2, 6; 7) \geq 17$. Now from Theorem \ref{theorem: F_v(a_1, ..., a_s; m - 1) geq m + p + 3} (p = 6) we obtain the lower bound

$F_v(a_1, ..., a_s; m - 1) \geq m + 9, \quad m \geq 8$.\\
To prove the upper bound consider the 18 vertex graph $\Gamma_3$ (Figure \ref{figure: H(3, 6; 7; 18) cap H(4, 5; 7; 18)}) with the help of which in \cite{SXP09} they prove the inequality $F_v(3, 6; 7) \leq 18$. In addition to the property $\Gamma_3 \overset{v}{\rightarrow} (3, 6)$ the graph $\Gamma_3$ also has the property $\Gamma_3 \overset{v}{\rightarrow} (4, 5)$. By repeating the arguments in the proof of Theorem \ref{theorem: wFv(m)(5)(m - 1) = ...}, Case 1. we see that from $\Gamma_3 \overset{v}{\rightarrow} (3, 6)$ and $\Gamma_3 \overset{v}{\rightarrow} (4, 5)$ it follows $\Gamma_3 \overset{v}{\rightarrow} \uni{8}{6}$. Since $\omega(\Gamma_3) = 6$, we obtain $\Gamma_3 \in \wH{8}{6}{7}$ and $\wFv{8}{6}{7} \leq \abs{\V(\Gamma_3)} = 18$. From this inequality and Theorem \ref{theorem: wFv(m)(p)(m - m_0 + q) leq wFv(m_0)(p)(q) + m - m_0} ($m_0 = 8, p = 6; q = 7$) it follows $\wFv{m}{6}{m - 1} \leq m + 10, \ m \geq 8$. At last, according to Theorem \ref{theorem: F_v(2_(m - p), p; q) leq F_v(a_1, ..., a_s; q) leq wFv(m)(p)(q)}

$F_v(a_1, ..., a_s; m - 1) \leq \wFv{m}{6}{m - 1} \leq m + 10, \quad m \geq 8$.\\
which finishes the proof of Theorem \ref{theorem: m + 9 leq F_v(a_1, ..., a_s; m - 1) leq m + 10, max set(a_1, ..., a_s) = 6}

\begin{figure}
	\centering
	\begin{subfigure}[b]{0.5\textwidth}
		\centering
		\includegraphics[height=240px,width=120px]{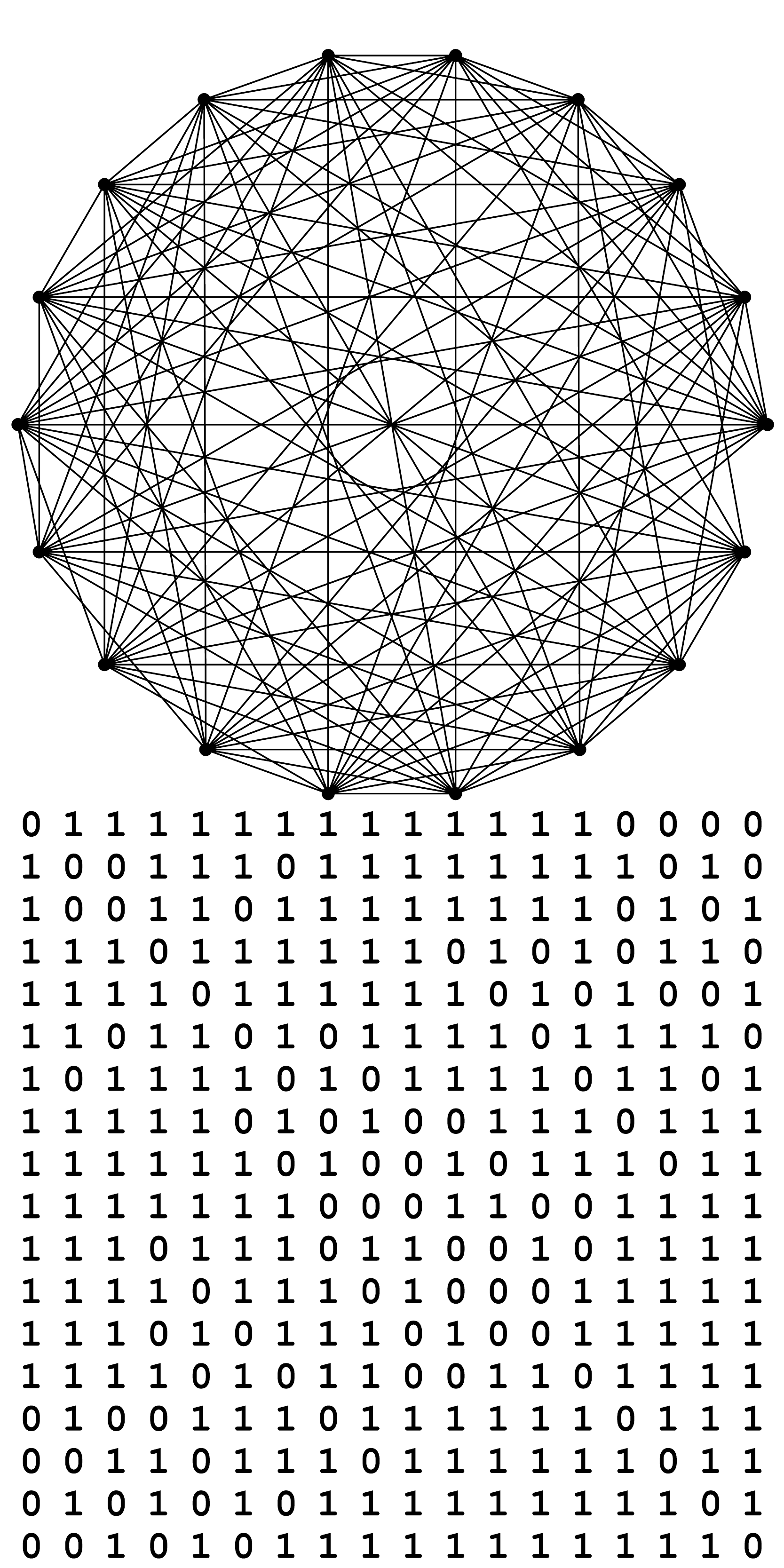}
		\caption*{\emph{$\Gamma_3$}}
		\label{figure: Gamma_3}
	\end{subfigure}
	\caption{Graph $\Gamma_3$, 18-vertex graph in $\mathcal{H}(3, 6; 7; 18)$ $\cap$ $\mathcal{H}(4, 5; 7; 18)$}
	\label{figure: H(3, 6; 7; 18) cap H(4, 5; 7; 18)}
\end{figure}

\appendix

\section{Obtaining the graph $\Gamma_2 \in \mathcal{H}(4, 5; 7; 17)$}

Consider the 18-vertex graph $\Gamma_3$ (Figure \ref{figure: H(3, 6; 7; 18) cap H(4, 5; 7; 18)}). As mentioned, this is the graph with the help of which in \cite{SXP09} they prove the inequality $F_v(3, 6; 7) \leq 18$. With the help of the computer we check that $\Gamma_3$ is maximal in $\mathcal{H}(4, 5; 7; 18)$. We use the following procedure to obtain other maximal graphs in $\mathcal{H}(4, 5; 7; 18)$.

\begin{procedure}
	\label{procedure: populate}
	Extending a set of maximal graphs in $\mathcal{H}(a_1, ..., a_s; q; n)$.
	
	1. Let $\mathcal{A}$ be a set of maximal graphs in $\mathcal{H}(a_1, ..., a_s; q; n)$.
	
	2. By removing edges from the graphs in $\mathcal{A}$, find all their subgraphs which are in $\mathcal{H}(a_1, ..., a_s; q; n)$. This way a set of non-maximal graphs in $\mathcal{H}(a_1, ..., a_s; q; n)$ is obtained.
	
	3. Add edges to the non-maximal graphs to find all their supergraphs which are maximal in $\mathcal{H}(a_1, ..., a_s; q; n)$. Extend the set $\mathcal{A}$ by adding the new maximal graphs.
\end{procedure}

Starting from a set containing a single element the graph $\Gamma_3$ and executing Procedure \ref{procedure: populate} we find 12 new maximal graphs  $\mathcal{H}(4, 5; 7; 18)$. Again, we execute Procedure \ref{procedure: populate} on the new set to find 110 more maximal graphs in $\mathcal{H}(4, 5; 7; 18)$. By removing one vertex from these graphs we obtain 17-vertex graphs, one of which is $\Gamma_2 \in \mathcal{H}(4, 5; 7; 17)$ shown on Figure \ref{figure: H(4, 5; 7; 17)}.


\section*{Acknowledgements}
This work was partially supported by the Sofia University Scientific Research Fund through Contract 144/2015.


\bibliographystyle{plain}

\bibliography{main}


\end{document}